\documentclass[letterpaper,11pt,oneside,reqno]{amsart}

\usepackage[backref=page,colorlinks=true,linkcolor=blue,citecolor=red]{hyperref}
\usepackage[alphabetic,nobysame]{amsrefs}

\usepackage[T1]{fontenc}
\usepackage{amsmath,amssymb,amsthm,amsfonts,mathtools}
\usepackage{graphicx,color}
\usepackage{upgreek}

\usepackage[all]{xy}

\xyoption{matrix}
\xyoption{arrow}
\usepackage{young}
\usepackage{hhline}
\usepackage{multirow}
\usepackage{shuffle}

\allowdisplaybreaks
\numberwithin{equation}{section}

\usepackage{tikz}
\usetikzlibrary{shapes,arrows,positioning,decorations.markings}

\usepackage{array}
\usepackage{adjustbox}
\usepackage{cleveref}
\usepackage{aliascnt}
\usepackage{enumerate}

\usepackage[most]{tcolorbox}

\usepackage[DIV=12]{typearea}
\usepackage{caption}
\theoremstyle{plain}
\newtheorem{lemma}{Lemma}[section]
\newaliascnt{theorem}{lemma}
\newtheorem{theorem}[theorem]{Theorem}
\aliascntresetthe{theorem}
\newaliascnt{proposition}{lemma}
\newtheorem{proposition}[proposition]{Proposition}
\aliascntresetthe{proposition}
\newaliascnt{corollary}{lemma}

\aliascntresetthe{corollary}
\theoremstyle{definition}
\newaliascnt{definition}{lemma}
\newtheorem{definition}[definition]{Definition}
\aliascntresetthe{definition}
\newaliascnt{remark}{lemma}
\newtheorem{remark}[remark]{Remark}
\aliascntresetthe{remark}
\newaliascnt{example}{lemma}

\aliascntresetthe{example}

\newcommand{\R}{\mathbb R}
\newcommand{\E}{\mathbb E}
\renewcommand{\P}{\mathbb P}
\renewcommand{\Re}{\operatorname{Re}}
\newcommand{\ssp}{\hspace{1pt}}

\begin{document}

\title{Perturbed Beta Corners Process}
\author{Leonid Petrov and Jiaming Xu}

\date{}


\begin{abstract}
	We introduce and study the perturbed $\beta$-corners
	process, a deformation of the classical $\beta$-corners
	process.  The latter is a probability measure on
	interlacing arrays of real numbers that, for the
	classical values $\beta=1,2,4$, describes the joint
	distribution of eigenvalues of principal submatrices of
	a Gaussian random matrix with the corresponding
	GOE/GUE/GSE symmetry.

	For classical $\beta$, the perturbed process arises
	from adding a deterministic diagonal matrix
	$A=\operatorname{diag}(a_1,\ldots,a_N)$ to a
	GOE/GUE/GSE matrix.  The eigenvalues of the resulting
	random matrix depend symmetrically on $a_1,\ldots,a_N$,
	but this symmetry does not extend to the whole corners
	process.

	We extend the construction to all $\beta > 0$ via
	multivariate Bessel functions, and analyze the
	crystallization ($\beta\to\infty$) of the resulting
	interlacing array, proving a law of large numbers and a
	central limit theorem in two regimes.  For fixed
	perturbation, the eigenvalue repulsion dominates, and
	the array crystallizes on the same lattice of
	polynomial-derivative roots, with the same discrete
	Gaussian free field fluctuations, as in the unperturbed
	case treated by Gorin and Marcus
	\cite{GorinMarcus2020}.  New phenomena appear in the
	second regime, where the $a_i$'s grow linearly
	in~$\beta$. Then the external source competes with the
	repulsion at leading order.  Here the array freezes on
	a deformed lattice characterized by a coupled system of
	optimality equations.  
	The fluctuations are governed by the same
	discrete Gaussian free field, now attached to the
	deformed lattice.

	The deformed lattice equations decouple
	in the special case of
	a single spike in the last coordinate. Then the
	deformed lattice is obtained explicitly by applying
	one shifted derivative
	$D_c f = f' + cf$ followed by iterated ordinary
	derivatives.
\end{abstract}

\maketitle

\setcounter{tocdepth}{1}
\tableofcontents
\setcounter{tocdepth}{4}

\section{Introduction}
\label{sec:intro}

\subsection{Overview}

We introduce and study the \emph{perturbed $\beta$-corners process},
a deformation of the Gaussian $\beta$-corners process depending on
parameters $\mathbf{a}=(a_1,\ldots,a_N)$.
For the classical values $\beta=1,2,4$,
the construction reduces to computing the joint eigenvalue distribution
of nested principal submatrices of $A + G$,
where $A=\operatorname{diag}(a_1,\ldots,a_N)$ is deterministic
and $G$ is a matrix
from the Gaussian Orthogonal/Unitary/Symplectic 
Ensemble (GOE/GUE/GSE),
respectively.
We extend the definition to all $\beta>0$ by analytic continuation,
in which the group-theoretic
Harish-Chandra--Itzykson--Zuber (HCIZ) integral is replaced by
multivariate Bessel functions, the symmetric eigenfunctions of the
Dunkl operators.
We then study the zero-temperature limit $\beta\to\infty$
at fixed particle number~$N$.

\medskip

The Gaussian $\beta$-ensemble (G$\beta$E) is the joint
eigenvalue density proportional to \begin{equation*} \prod_{1\le
i<j\le N}|\lambda_i-\lambda_j|^\beta\,e^{-(1/2)\sum
\lambda_i^2}, \end{equation*} which for $\beta=1,2,4$
corresponds to the eigenvalues of GOE/GUE/GSE matrices.

\begin{figure}[ht]
	\centering
	\begin{adjustbox}{max width=\textwidth}
	\begin{tikzpicture}[x=1cm,y=1cm,line join=round]
		\fill[black!12] (0,0.45) rectangle (2.25,2.7);
		\draw (0,0) rectangle (2.7,2.7);
		\draw[very thick] (0,0.45) rectangle (2.25,2.7);
		\begin{scope}[shift={(4.6,-0.15)}]
			\foreach \xm/\xa/\xb in
				{0.95/0.3/1.25, 1.7/1.25/2.45, 2.9/2.45/3.2, 3.65/3.2/4.4, 5.2/4.4/5.6}{
				\draw[black!30] (\xm,1.15) -- (\xa,1.9);
				\draw[black!30] (\xm,1.15) -- (\xb,1.9);
			}
			\foreach \x/\i in {0.3/6, 1.25/5, 2.45/4, 3.2/3, 4.4/2, 5.6/1}{
				\fill (\x,1.9) circle (1.7pt);
				\node[above=0.4mm] at (\x,1.9) {$\lambda_{\i}$};
			}
			\foreach \x/\i in {0.95/5, 1.7/4, 2.9/3, 3.65/2, 5.2/1}{
				\fill (\x,1.15) circle (1.7pt);
				\node[below=0.4mm] at (\x,1.15) {$\mu_{\i}$};
			}
		\end{scope}
	\end{tikzpicture}
	\end{adjustbox}
	\caption{Cutting the upper left $(N-1)\times(N-1)$ corner.
	The eigenvalues $\boldsymbol{\mu}$ of the corner interlace with the
	eigenvalues $\boldsymbol{\lambda}$ of the matrix.}
	\label{fig:corner_cutting}
\end{figure}
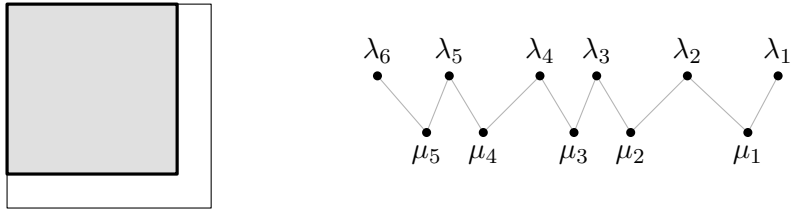

The \emph{corners process} describes the joint distribution of
eigenvalues of nested principal submatrices
(\Cref{fig:corner_cutting}).  For $\beta=2$ (GUE), this
structure was identified by Baryshnikov~\cite{Baryshnikov_GUE2001} and
Johansson--Nordenstam~\cite{johansson2006eigenvalues}; for general
$\beta>0$, the Gaussian case was constructed by Gorin and
Shkolnikov~\cite{GorinShkolnikov2014} via multilevel Dyson Brownian
motions, and the Jacobi case by Borodin and Gorin
\cite{BorodinGorin2015}, who also proved that the global fluctuations of
the $\beta$-corners process as $N\to\infty$ and $\beta$ fixed are
governed by the Gaussian free field.  An alternative approach to
$\beta$-corners processes, based on beta integrals over Rayleigh
triangles, was developed by Neretin~\cite{Neretin2003Rayleigh}.  We note
that there is a tridiagonal random matrix model for G$\beta$E for all
$\beta>0$ introduced by Dumitriu and Edelman~\cite{dumitriu2002matrix},
which, however, cannot be used to construct the corners process, even in
the unperturbed case~\cite{GorinKleptsyn2024}.

In the log-gas picture, $\beta$ plays the role of the
inverse temperature, so that $\beta\to 0$ is the
high-temperature regime and $\beta\to\infty$ the
zero-temperature one.  The former was studied by
Benaych-Georges, Cuenca, and
Gorin~\cite{benaych2022matrix} and by Cuenca and
Do\l{}ega~\cite{cuenca2025discrete}.  The
zero-temperature regime (leading to crystallization of
the eigenvalues) of the $\beta$-corners process was
established by Gorin and Marcus~\cite{GorinMarcus2020}:
as $\beta\to\infty$ the eigenvalue array freezes on a
deterministic lattice given by iterated polynomial
derivatives, with discrete Gaussian
free field fluctuations on top of it.  This has been
further developed by Gorin and
Kleptsyn~\cite{GorinKleptsyn2024} (the G$\infty$E and
Airy$_\infty$ ensembles), and extended to Laguerre
corners by Lerner-Brecher~\cite{LernerBrecher2023}.  In
the language of finite free probability of Marcus,
Spielman, and
Srivastava~\cite{MarcusSpielmanSrivastava2022} and
Marcus~\cite{Marcus2021}, passing to eigenvalues of a corner
one size less
at
$\beta=\infty$ is the finite free projection, that is,
differentiation of the characteristic polynomial.

\medskip

Adding a deterministic matrix to a random one --- an
\emph{external source} or \emph{spike} --- is also
classical.  Baik, Ben Arous, and
P\'ech\'e~\cite{BBP2005phase} found, for sample
covariance matrices, the phase transition
at which a sufficiently strong spike makes the largest
eigenvalue detach from the bulk; the additive
counterpart, for Hermitian matrices shifted by a
finite-rank deterministic matrix, is due to
P\'ech\'e~\cite{Peche2006}.  This has led to
numerous follow-up works, which we do not survey here.

\medskip

The idea of adding a deterministic matrix was extended to
random matrix corners.  For $\beta=2$, Ferrari and
Frings~\cite{Ferrari2014PerturbedGUE} considered the
matrix diffusion $H(t)+tA$, where $H$ is the Hermitian
Brownian motion started at the null matrix, so that at
$t=1$ one recovers $A+G$.  Under $H(t)$, the eigenvalues
perform the Dyson Brownian motion
\cite{dyson1962brownian}.  Ferrari and Frings computed
the correlation kernel of the corners process of
$H(t)+tA$ at a fixed time, and identified the resulting
corners distribution with the fixed-time distribution of
a Warren process \cite{warren2005dyson} in which all
Brownian motions at level $k$ carry the same drift $a_k$
and reflect off those at level $k-1$.  The same
correlation kernel was obtained independently by Adler,
van Moerbeke, and Wang~\cite{adler2013random}.  Petrov and
Tikhonov~\cite{PetrovTikhonov2019} studied the interplay
between the symmetry of the level-$k$ eigenvalue
distribution in $a_1,\ldots,a_k$ and the failure of this
symmetry for the whole corners distribution, via certain
Markov jump operators acting on individual particles.
For general~$\beta$, 
Bloemendal and Vir\'ag~\cite{bloemendal2013limits} and
Lamarre and
Shkolnikov~\cite{LamarreShkolnikov2019} studied the edge
of the single-level Gaussian $\beta$-ensemble with one
spike (through the tridiagonal model of
\cite{dumitriu2002matrix}).  Two recent works rely on the
same machinery as the present paper, namely Dunkl
operators acting on multivariate Bessel generating
functions: Keating and Xu~\cite{KeatingXu2024} consider
the $\beta$-addition of a Gaussian and finitely many
Laguerre $\beta$-ensembles (the same operation as ours,
except that the added spectrum is random instead of the
fixed diagonal~$A$) and prove that its edge limit is the
Airy$_\beta$ process.  Gorin, Xu, and
Zhang~\cite{gorin2024airy} characterize the Airy$_\beta$
line ensemble and identify it as the edge limit of the
unperturbed G$\beta$E corners process.  In both cases the
limit is $N\to\infty$ at fixed~$\beta$.  We instead fix
$N$, let $\beta\to\infty$, and follow the perturbation
through the whole interlacing array.

\subsection{Main results}
\label{sub:main_results}

Let us outline our main contributions.

\medskip

First, for the classical values $\beta\in\{1,2,4\}$, we
derive the level-by-level Markov transition kernels
(\emph{links}) of the corners process of $C = A + G$
(where $A$ is fixed and $G$ is random GOE/GUE/GSE) by a
change-of-variables calculation
(\Cref{prop:cond_density_beta124}).  We then extend the
density of the perturbed corners process to all $\beta>0$
via multivariate Bessel functions (see
Opdam~\cite{Opdam1993} and Heckman
and Opdam~\cite{HeckmannOpdam1997}), which replace the
classical HCIZ integrals available only for
$\beta=1,2,4$.  The resulting definition of the perturbed
$\beta$-corners process
(\Cref{def:perturbed_beta_corners}) recovers the
classical $\beta$ formulas and preserves the interlacing
structure.

\medskip 

Second, we analyze the fixed $N$, $\beta\to\infty$ limit
of the perturbed $\beta$-corners process
with a fixed top row
in two regimes:
\begin{enumerate}[$\bullet$]
\item \emph{Fixed perturbation} (all $a_i$ fixed), when the
eigenvalue repulsion dominates.
The corners eigenvalues crystallize on the same lattice as in the
unperturbed model of Gorin and Marcus~\cite{GorinMarcus2020}, and the
perturbation becomes invisible
(\Cref{thm:crystallization_fixed}).
The rescaled fluctuations converge to the
discrete Gaussian free field described in~\cite{GorinMarcus2020}.
\item \emph{Scaled perturbation}
($a_i=\frac{\beta}{2}\mathsf{a}_i$), when
the external source competes with the repulsion.
This regime 
produces a \emph{deformed} crystal lattice whose positions
are the unique solution of a coupled system of
optimality equations
(\Cref{def:deformed_lattice} and
\Cref{thm:crystallization_scaled}).
The eigenvalue fluctuations converge to a deformed discrete
Gaussian free field (\Cref{thm:CLT_scaled}).
When the perturbation is constant
($\mathsf{a}_1=\cdots=\mathsf{a}_N$),
we recover the same limits as in the unperturbed case.
\end{enumerate}

For the asymptotic analysis as $\beta\to\infty$, a key
simplification is that the multivariate Bessel function
of the top row appears in the perturbed G$\beta$E density
and, inverted, in the perturbed links, so it cancels from
their product. The joint density is then governed by an
explicit log-gas energy (which we refer to as the
effective Hamiltonian).

In both cases (fixed and scaled perturbation), 
the array of the perturbed
$\beta$-corners process crystallizes on the minimizer of
the effective Hamiltonian. In the unperturbed
case, treated in~\cite{GorinMarcus2020}, the minimizer is
explicit: it is the lattice of roots of iterated
polynomial derivatives. The deformed lattice admits no
such formula in general (but does so for a single spike, see \Cref{sec:shifted_derivative}).
We proceed by showing the strict convexity of
the effective Hamiltonian (and, moreover, that it blows
up at the boundary of the Gelfand--Tsetlin polytope);
see \Cref{prop:strict_convexity} and 
\Cref{lem:boundary_blowup}.
The unique global minimizer is then the deformed lattice.
Once the convexity is established, the law of large
numbers and the central limit theorem follow by a
standard Laplace-type argument
\cite{Wong2001}*{Chapter~9}. See
\Cref{rmk:convexity_vs_GM} for a detailed comparison of
our technique with the approach of
\cite{GorinMarcus2020}.

In
\Cref{sec:up_transitions}, we take
the same zero-temperature limit 
in the other direction,
of the conditional distribution of 
the top row $\boldsymbol{\lambda}$ given the lower
row $\boldsymbol{\mu}$ (cf.~\Cref{fig:corner_cutting}), 
in both regimes of fixed and scaled perturbation.
These limits of upward transitions are powered by
Rodrigues-type
weighted 
derivatives $w^{-1}(wf)'$, where $w$ depends on the
regime and can be Gaussian, exponential, or constant.

\subsection{Further directions}
\label{sec:further_directions}

We work at fixed $N$ throughout; the $N\to\infty$ limit
of the perturbed $\beta$-corners process is out of scope
of this paper.  In particular, we do not consider a
possible BBP-type transition for a perturbed version of
the Airy$_\beta$ line ensemble.  For a single level and a
single spike this transition is
known~\cite{bloemendal2013limits},~\cite{LamarreShkolnikov2019},
and for finitely many spikes it is known at
$\beta=1,2,4$~\cite{bloemendal2016limits}.  Extending
this to several spikes at general $\beta$, and to the
whole interlacing array rather than its top row alone,
remains open.  We also do not know whether the
distributional symmetry of the perturbed GUE corners
process discovered in~\cite{PetrovTikhonov2019} extends
to general~$\beta$; their construction rests on the
uniform Gibbs property of the GUE, which nontrivially
deforms for general~$\beta$.  Two further extensions are
a perturbed \emph{Jacobi} process, whose eigenvalues are
confined to a compact interval and which carries
additional boundary parameters, and the high-temperature
regime $\beta\to0$, where the repulsion weakens instead
of freezing the array. Matrix addition in that regime is
treated by Benaych-Georges, Cuenca, and
Gorin~\cite{benaych2022matrix}.

\subsection{Notation}

We collect the main notation used throughout the paper.
The parameter $\beta>0$ is the inverse temperature
(with $\beta=1,2,4$ corresponding to the GOE, GUE, and GSE, respectively),
and $N\ge 1$ is the matrix size, which remains fixed
throughout the asymptotic analysis.

Bold letters denote ordered tuples:
$\boldsymbol{\lambda}=(\lambda_1\ge\cdots\ge\lambda_N)$,
$\boldsymbol{\mu}=(\mu_1\ge\cdots\ge\mu_{N-1})$,
$\mathbf{a}=(a_1,\ldots,a_N)$ (perturbation parameters),
$\mathbf{z}=(z_1>\cdots>z_N)$ (fixed top row).
The interlacing array has levels
$\mathbf{y}^k=(y^k_1\ge\cdots\ge y^k_k)$ for $k=1,\ldots,N$,
and we write $\mathbf{Y}=(\mathbf{y}^1,\ldots,\mathbf{y}^{N-1})$
for the collection of lower levels.

Interlacing $\boldsymbol{\mu}\prec\boldsymbol{\lambda}$
means $\lambda_1\ge\mu_1\ge\lambda_2\ge\cdots\ge\mu_{N-1}\ge\lambda_N$.
The arrays $\mathbf{Y}$ satisfying
$\mathbf{y}^1\prec\cdots\prec\mathbf{y}^{N-1}\prec\mathbf{z}$
with the top row $\mathbf{z}$ held fixed form the
\emph{Gelfand--Tsetlin polytope}.
We write $\mathbf{1}_{A}$ for the indicator of an event or condition~$A$.
For a tuple $\mathbf{w}$ of distinct reals we write
$V(\mathbf{w})=\prod_{i<j}|w_i-w_j|$ for its Vandermonde.
The notation $\coloneqq$ stands for ``is defined as.''

Key objects defined in the text:
$B_{\mathbf{a}}(\mathbf{z};\beta)$ is the multivariate Bessel function
(\Cref{sec:bessel_gf}),
$f_{\mathbf{a}}$ is the perturbed G$\beta$E density
\eqref{eq:density_perturbed_GbetaE},
$\Lambda_{\mathbf{a}}$ is the perturbed link
\eqref{eq:perturbed_corners_conditional},
and $F_{\mathbf{a}}=f_{\mathbf{a}}\cdot\Lambda_{\mathbf{a}}$
is the joint density of the perturbed corners process
(\Cref{def:perturbed_beta_corners}).
In \Cref{sec:beta_inf_trivial,sec:beta_inf_scaled} we reserve sans-serif letters for
deterministic objects, italic ones remaining random:
$\mathsf{y}^k_i$ denotes the unperturbed deterministic lattice
(\Cref{def:deterministic_lattice}),
and $\bar{\mathsf{y}}^k_i$ the deformed lattice
(\Cref{def:deformed_lattice}) arising under
the scaled perturbation $a_i=\frac{\beta}{2}\mathsf{a}_i$
with $\mathsf{a}=(\mathsf{a}_1,\ldots,\mathsf{a}_N)$.
The capital letters denote the whole arrays: the $\infty$-corners lattice
$\mathsf{Y}=\{\mathsf{y}^k_i\}$ and the deformed lattice
$\bar{\mathsf{Y}}=\{\bar{\mathsf{y}}^k_i\}$.
The shifted derivative operator is
$D_c f = f' + cf$~\eqref{eq:shifted_deriv_operator}.

\subsection{Outline}

The remainder of the paper is organized as follows.
\Cref{sec:beta124} derives the perturbed corners process
for $\beta=1,2,4$ from the matrix model, and
\Cref{sec:beta_general} extends it to all $\beta>0$
via multivariate Bessel functions.
The zero-temperature analysis fills the last three sections.
A fixed perturbation leaves the unperturbed lattice in place
(\Cref{sec:beta_inf_trivial}); a perturbation growing linearly
in $\beta$ deforms it (\Cref{sec:beta_inf_scaled}).
\Cref{sec:up_transitions} considers zero-temperature
the behavior of 
the upper row given the lower row in the 
perturbed $\beta$-corners process, 
and identifies the limit in both
regimes as a weighted derivative.

\subsection*{Acknowledgements}

We are grateful to Vadim Gorin for helpful discussions.
LP was partially supported by the NSF grant DMS-2153869
and by the Simons Foundation Travel Support for Mathematicians
award SFI-MPS-TSM-00013561.

\subsection*{AI statement.} 
All the ideas in this paper are due to the authors.
Some proof details in \Cref{sec:beta_inf_scaled,sec:up_transitions}, as well as
parts of the text, were produced with AI assistance (ChatGPT Pro GPT-5.6 by OpenAI and
Claude Code Opus~4.8 and Fable~5 by Anthropic).
We checked every statement and every
argument, and the responsibility for the correctness of the paper is ours
alone.

\section{Derivation for $\beta=1,2,4$ from Gaussian matrix models}
\label{sec:beta124}

\subsection{Setup}
Consider the matrix addition 
\begin{equation}
\label{eq:perturbed_GOE_GUE_GSE}
		C=A+G,
\end{equation}
where $A=\operatorname{diag}(a_{1},\ldots,a_{N})$ with
fixed deterministic eigenvalues $a_{1},\ldots,a_{N}\in
\R$, and $G=\frac{X+X^{*}}{2}$ is the $N\times N$
GOE/GUE/GSE matrix.  Here $X^*$ denotes the transpose for
the GOE ($\beta=1$, real entries), the conjugate
transpose for the GUE ($\beta=2$, complex entries), and
the quaternionic conjugate transpose for the GSE
($\beta=4$, quaternionic entries).  The matrix $X$ is
$N\times N$, with independent identically distributed
real/complex/qua\-ter\-nionic Gaussian entries,
normalized so that $\Re(X_{ij})\sim \mathcal{N}(0,1)$ for
all $i,j=1,\ldots,N$.

Denote the eigenvalues of $C=C_{N}$ by $\boldsymbol{\lambda}=(\lambda_{1}\ge \cdots \ge\lambda_{N})$, and the eigenvalues of $C_{N-1}$,
the $(N-1)\times (N-1)$ upper left corner of $C$, by 
$\boldsymbol{\mu}=(\mu_{1}\ge \cdots \ge\mu_{N-1})$.
The goal of this section is to derive the conditional density
$h(\boldsymbol{\lambda}\mid \boldsymbol{\mu})$ of $\boldsymbol{\lambda}$ given $\boldsymbol{\mu}$
(\Cref{prop:cond_density_beta124} below).

\subsection{Diagonalization}

Let us diagonalize $C_{N-1}$ by some $U\in U_{\beta}(N-1)$ (here
$U_{\beta}(n)$ stands for the Orthogonal/Unitary/Unitary Symplectic
group of size $n$, respectively, for $\beta=1,2,4$), and conjugate
$C_{N}$ by $\begin{bmatrix} U& 0\\ 0& 1
\end{bmatrix}\in U_{\beta}(N)$.
We get
\begin{equation*}
    \tilde{C}
		\coloneqq \begin{bmatrix}
				U&0\\
				0&1
		\end{bmatrix}C\begin{bmatrix}
				U^{*}&0\\
				0&1
		\end{bmatrix}
		=\begin{bmatrix}
        \mu_{1}&0&&0& P_{1}\\
        0&\ddots &&0&P_{2}\\
        &&\ddots &&\\
        &&&\mu_{N-1}&P_{N-1}\\
        {P}^*_{1}&\ldots &\ldots &{P}^*_{N-1}&\xi_{N}
    \end{bmatrix}.
\end{equation*}
Here the last diagonal entry is $\xi_N=c_{NN}\sim \mathcal{N}(a_N,1)$, 
and
\begin{equation}
	\label{eq_Pi}
	P_i=\sum_{j=1}^{N-1}u_{ij}c_{jN},\qquad i=1,\ldots,N-1.
\end{equation}
In vector form, \eqref{eq_Pi} reads
$(P_{1},\ldots,P_{N-1})^{t}=U\,(c_{1N},\ldots,c_{N-1,N})^{t}$:
the $P_{i}$ are the entries of the last column of $C$ above the
diagonal, written in the eigenbasis of the corner~$C_{N-1}$.
Their joint distribution can be understood as follows.
First, $A$ is diagonal, so the
off-diagonal entries of $C$ coincide with those of~$G$, and their
joint law does not involve the perturbation parameters.
Second, in the GOE/GUE/GSE the entries on and above the diagonal are
independent, so these $N-1$ entries are independent of the
corner~$C_{N-1}$; since $U$ is a function of $C_{N-1}$ alone, they are
independent of~$U$ as well.
Third, the joint law of $P_1,\ldots,P_{N-1}$
is invariant under $U_{\beta}(N-1)$, so
multiplying by $U$ does not change it.
Consequently, the $P_{i}$'s are independent, with the
real/complex/quaternionic Gaussian distribution normalized so that
$\Re(P_{i})\sim \mathcal{N}(0,\frac{1}{2})$, exactly as in the
unperturbed case.

\subsection{Characteristic polynomial}

The eigenvalue equation $\operatorname{det}(x-\tilde{C})=0$,%
\footnote{For $\beta=4$ the matrix $\tilde{C}$ is quaternionic, and
$\operatorname{det}$ denotes the reduced (Moore) determinant, that is,
the monic degree-$N$ polynomial $\prod_{i=1}^{N}(x-\lambda_{i})$.
It obeys the same block Schur-complement factorization as in the real
and complex cases, so the derivation below is uniform in
$\beta=1,2,4$.}
divided by 
the product 
$\pm(x-\mu_{1})\cdots (x-\mu_{N-1})$,
takes the form
\begin{equation}\label{eq_polynomialeqn}
    \sum_{i=1}^{N-1}\frac{\xi_{i}}{x-\mu_{i}}-(x-\xi_{N})=0.
\end{equation}
The coefficients $\xi_i$, $i=1,\ldots,N-1$,
are expressed through $P_i$ \eqref{eq_Pi} as
\begin{equation*}
	\xi_i= P_i P_i^*=\|P_i\|^2\sim \sum_{j=1}^{\beta}\mathcal{N}(0,\tfrac{1}{2})^2\sim\tfrac{1}{2}\chi_\beta^2,
\end{equation*}
where $\chi_\beta^2$ is the chi-squared distribution with $\beta$ degrees of freedom.

Since $\lambda_{i}$'s are all the $N$ roots of \eqref{eq_polynomialeqn},
we have the interlacing $\boldsymbol{\mu} \prec \boldsymbol{\lambda}$
thanks to the behavior of the left-hand side
around its poles $\mu_{i}$.
Moreover, \eqref{eq_polynomialeqn} can be written in the following equivalent forms:
\begin{align}\label{eq_polynomialeqn2}
    \sum_{i=1}^{N-1}\frac{\xi_{i}}{x-\mu_{i}}
    &= -\frac{(x-\lambda_{1})\cdots (x-\lambda_{N})}{(x-\mu_{1})\cdots (x-\mu_{N-1})}+x-\xi_{N},\qquad\text{or}\\\label{eq_polynomialeqn3}
    \sum_{i=1}^{N-1}\xi_{i}\prod_{j\ne i}(x-\mu_{j})
    &=
		\prod_{i=1}^{N-1}(x-\mu_{i})(x-\xi_{N})-(x-\lambda_{1})\cdots (x-\lambda_{N}).
\end{align}
Conditioning on $\boldsymbol{\mu}$, 
identities \eqref{eq_polynomialeqn2} or \eqref{eq_polynomialeqn3}
define
a bijection
\begin{equation*}
(\xi_{1},\ldots,\xi_{N-1},\xi_N)\longleftrightarrow (\lambda_{1},\ldots,\lambda_{N}).
\end{equation*}
To find the conditional density of $\boldsymbol{\lambda}$ given $\boldsymbol{\mu}$,
it remains to calculate the Jacobian $\left|\frac{\partial \xi_{i}}{\partial \lambda_{j}}\right|_{i,j=1,\ldots,N}$.

\subsection{Jacobian}
\label{sec:Jacobian}

Setting $x= \mu_{i}$ for $i=1,\ldots,N-1$ in \eqref{eq_polynomialeqn3}, we have
\begin{equation}\label{eq_xi1}
		\xi_{i}=-\frac{(\mu_{i}-\lambda_{1})\cdots
		(\mu_{i}-\lambda_{N})}{\prod_{j\ne i}(\mu_{i}-\mu_{j})}\implies
		\frac{\partial \xi_{i}}{\partial \lambda_j}=\frac{\prod_{d\ne
		j}(\mu_{i}-\lambda_{d})}{\prod_{d\ne i }(\mu_{i}-\mu_{d})},
\end{equation}
where $j=1,\ldots ,N$.

For $i=N$, by comparing the constant terms of the two polynomials
in~$x$ in \eqref{eq_polynomialeqn3}, we get
\begin{align}
	\nonumber
	&
	(-1)^{N+1}\prod_{i=1}^{N}\lambda_{i}+(-1)^{N}\xi_{N}\prod_{i=1}^{N-1}\mu_{i}
	=
	\sum_{i=1}^{N-1}\xi_{i}(-1)^{N-2}\prod_{j\ne i}\mu_{j}
	\\
	\nonumber
	&\hspace{40pt}
	\implies \xi_{N}=\frac{\sum_{i=1}^{N-1}\xi_{i}\prod_{k\ne
	i}\mu_{k}+\prod_{i=1}^{N}\lambda_{i}}{\prod_{i=1}^{N-1}\mu_{i}}\\
	&\hspace{40pt}
	\implies \frac{\partial \xi_{N}}{\partial
	\lambda_{j}}=\frac{\sum_{i=1}^{N-1}\frac{\partial
	\xi_{i}}{\partial \lambda_{j}}\prod_{k\ne i}\mu_{k}+\prod_{d\ne
	j}\lambda_{d}}{\prod_{i=1}^{N-1}\mu_{i}}.
	\label{eq:Jac_N}
\end{align}
Here also $j=1,\ldots ,N$. 

Note that the terms in \eqref{eq:Jac_N}
involving $\frac{\partial \xi_{i}}{\partial \lambda_{j}}$
are linear combinations of the previous rows of the Jacobian matrix, and can be ignored
when calculating the determinant.
Thus, we have
\begin{equation*}
    \begin{split}
        &\operatorname{det}\left(\frac{\partial \xi_{i}}{\partial \lambda_j}\right)_{i,j=1,\ldots,N}\\
        &\hspace{10pt}=
				\frac{1}{\prod_{k=1}^{N-1}\prod_{d\ne k}(\mu_{k}-\mu_{d})}
				\times
				\prod_{k=1}^{N-1}\prod_{d=1}^{N}(\mu_{k}-\lambda_{d})
				\times\frac{\prod_{d=1}^{N}\lambda_{d}}{\prod_{k=1}^{N-1}\mu_{k}}\operatorname{det}\begin{bmatrix}
            \frac{1}{\mu_{1}-\lambda_{1}}&&\ldots&&\frac{1}{\mu_{1}-\lambda_{N}}\\[3pt]
            \ldots&&&&\ldots\\[3pt]
            \ldots&&&&\ldots\\[3pt]
            \frac{1}{\mu_{N-1}-\lambda_{1}}&&\ldots&&\frac{1}{\mu_{N-1}-\lambda_{N}}\\[6pt]
            \frac{1}{\lambda_{1}}&&\ldots&&\frac{1}{\lambda_{N}}
        \end{bmatrix}.
    \end{split}
\end{equation*}
The last determinant above is the Cauchy determinant, so it is equal to
\begin{equation*}
(-1)^{N-1}\frac{\prod_{1\le i<j\le N}(\lambda_{i}-\lambda_{j})(\mu_{j}-\mu_{i})}{\prod_{i=1}^{N}\prod_{j=1}^{N}(\lambda_{i}-\mu_{j})},
\end{equation*}
where we set $\mu_{N}=0$ for convenience. Hence, up to
$\boldsymbol{\mu}$-dependent factors,
the absolute value of the Jacobian
is the Vandermonde:
\begin{equation*}
\left|\operatorname{det}\left(\frac{\partial \xi_{i}}{\partial\lambda_{j}}\right)_{i,j=1,\ldots,N}\right|
\propto
\prod_{1\le i<j\le N}(\lambda_{i}-\lambda_{j}).
\end{equation*}

\subsection{Conditional density}

We can now put everything together 
for classical $\beta$:
\begin{proposition}
	\label{prop:cond_density_beta124}
	For $\beta=1,2$, or $4$,
	let $\boldsymbol{\lambda}=(\lambda_{1}\ge \cdots \ge \lambda_{N})$
	be the eigenvalues of the perturbed $N\times N$ GOE/GUE/GSE matrix $C=A+G$ 
	\eqref{eq:perturbed_GOE_GUE_GSE},
	and let $\boldsymbol{\mu}=(\mu_{1}\ge \cdots \ge \mu_{N-1})$
	be the eigenvalues of its $(N-1)\times (N-1)$ principal corner.
	Then
	the conditional density of
	$\boldsymbol{\lambda}$ given $\boldsymbol{\mu}$ is
	\begin{equation}\label{eq_cond_density_beta124}
		h(\boldsymbol{\lambda} \mid \boldsymbol{\mu}) \propto
		\exp\left(-\frac{1}{2}\sum_{i=1}^{N}\lambda_i^2\right)
		\exp\left(a_N \sum_{i=1}^{N}\lambda_i\right)
		\prod_{1\le i<j\le N}(\lambda_i - \lambda_j)
		\prod_{i=1}^{N-1}\prod_{j=1}^{N}|\mu_i - \lambda_j|^{\beta/2-1}
		\mathbf{1}_{\boldsymbol{\lambda} \succ \boldsymbol{\mu}},
	\end{equation}
	where the proportionality constant depends only on $\boldsymbol{\mu}$ and the parameters $a_i$.
\end{proposition}
\begin{proof}
	The variables $(\xi_1, \ldots, \xi_{N-1}, \xi_N)$ are conditionally independent given
	$\boldsymbol{\mu}$, 
	with joint density
	\begin{equation*}
		f(\boldsymbol{\xi}) \propto
		\left[\prod_{i=1}^{N-1} \xi_i^{\beta/2-1} e^{-\xi_i}\right]
		e^{-\xi_N^2/2} e^{a_N \xi_N}.
	\end{equation*}
	By the first identity in \eqref{eq_xi1}, $\xi_i$
	vanishes precisely when $\mu_i$ coincides with one of the~$\lambda_j$.
	Hence $\xi_i>0$ for all $i=1,\ldots,N-1$ iff the interlacing
	$\boldsymbol{\lambda} \succ \boldsymbol{\mu}$ is strict. The
	weak boundary, where $\xi_i=0$ for some $i$, is Lebesgue-null and
	irrelevant for the density.
	The same identity gives
	$\xi_i^{\beta/2-1} \propto \prod_{j=1}^{N}|\mu_i - \lambda_j|^{\beta/2-1}$,
	where we absorbed $\boldsymbol{\mu}$-dependent terms into the normalization constant.
	Comparing the coefficients by $x^{N-1}$ in \eqref{eq_polynomialeqn3}
	gives
	\begin{equation}\label{eq_xi2}
		\xi_N = \sum_{i=1}^{N}\lambda_i - \sum_{i=1}^{N-1}\mu_i,
	\end{equation}
	so $e^{a_N \xi_N} \propto \exp(a_N \sum_i \lambda_i)$.
	Finally,
	comparing the coefficients by $x^{N-2}$ in \eqref{eq_polynomialeqn3} yields
	the remaining terms:
	\begin{equation}\label{eq_xi3}
		\sum_{i=1}^{N-1}\xi_i =
		-\sum_{1\le i<j\le N}\lambda_i \lambda_j
		+
		\left(\sum_{i=1}^{N}\lambda_i\right)\left(\sum_{i=1}^{N-1}\mu_i\right)
		+
		\sum_{1\le i<j\le N-1}\mu_i \mu_j-
		\left( \sum_{i=1}^{N-1} \mu_i\right)^2.
	\end{equation}
	Combined with \eqref{eq_xi2}, 
	we have
	\begin{equation*}
		-\sum_{i=1}^{N-1}\xi_i-\frac{\xi_N^2}{2}=
		-\frac{1}{2}\sum_{i=1}^{N}\lambda_i^2+
		\frac{1}{2}\sum_{i=1}^{N-1}\mu_i^2.
	\end{equation*}
Multiplying by the Jacobian $\prod_{i<j}(\lambda_i - \lambda_j)$ from
\Cref{sec:Jacobian} completes the proof. 
\end{proof}

\begin{remark}
	\label{rmk:collecting_mu_terms_for_full_density}
	Collecting all the terms depending only on $\boldsymbol{\mu}$
	in the proof of \Cref{prop:cond_density_beta124},
	one can write the conditional density \eqref{eq_cond_density_beta124} more explicitly as
	\begin{equation*}
		\begin{split}
			h(\boldsymbol{\lambda} \mid \boldsymbol{\mu}) &=
			\frac{e^{-a_N^2/2}}{\Gamma(\beta/2)^{N-1}\sqrt{2\pi}}
			\prod_{i=1}^{N-1}
			\prod_{j=1}^{N}|\mu_i - \lambda_j|^{\beta/2-1}
			\prod_{1\le i<j\le N-1}(\mu_i - \mu_j)^{1-\beta}
			\prod_{1\le i<j\le N}(\lambda_i - \lambda_j)
			\\
			&\hspace{40pt}
			\times
			\exp\left( a_N\sum_{i=1}^{N}\lambda_i - a_N\sum_{i=1}^{N-1}\mu_i
			-\frac{1}{2}\sum_{i=1}^{N}\lambda_i^2+\frac{1}{2}\sum_{i=1}^{N-1}\mu_i^2 \right)\mathbf{1}_{\boldsymbol{\lambda} \succ \boldsymbol{\mu}}.
		\end{split}
	\end{equation*}
\end{remark}

\section{Perturbed corners process for general $\beta$}
\label{sec:beta_general}

We now extend the model from \Cref{sec:beta124} to general $\beta > 0$
via analytic continuation.
The key ingredients are well-known, and include the Gaussian $\beta$-ensembles
(G$\beta$E)~\cite{dumitriu2002matrix},
the $\beta$-corners process~\cite{GorinShkolnikov2014},
and the $\beta$-addition operation on random matrix spectra
\cite{GorinMarcus2020}, 
\cite{benaych2022matrix},
\cite{KeatingXu2024}, \cite{gorin2024airy}.
For general $\beta$, there is no underlying matrix model for which one can
consider the corners, so the construction relies on
multivariate Bessel generating functions.

\subsection{Multivariable Bessel generating functions}
\label{sec:bessel_gf}

Here we briefly recall the necessary properties of multivariate Bessel functions
which are $\beta$-deformations of the Harish-Chandra-Itzykson-Zuber (HCIZ) 
integrals~\cite{HarishChandra1957},~\cite{itzykson1980planar}.

\begin{remark}
	In the literature on Jack polynomials, multivariate Bessel functions, and Dunkl operators it is
	customary to use the parameters $\theta = \beta/2$ or 
	$\alpha=2/\beta$ instead of $\beta$.
	In this paper we keep the
	notation $\beta$ throughout for consistency with the random matrix formulas.
\end{remark}

Let $\mathbf{a}=(a_1,\ldots,a_N)$ be an $N$-tuple of real numbers.
The \emph{multivariate Bessel function}, denoted by $B_{\mathbf{a}}(\mathbf{z};\beta)$,
where $\mathbf{z}=(z_1,\ldots,z_N)$ and $\beta>0$,
can be defined as
the unique
(up to multiplication by a constant)
entire eigenfunction, symmetric in $z_1,\ldots,z_N$, of the
symmetric Dunkl operators~\cite{Opdam1993}
\begin{equation*}
	\mathcal{P}_k \coloneqq \mathcal{D}_1^k + \cdots + \mathcal{D}_N^k,\qquad
	\mathcal{D}_i \coloneqq \frac{\partial}{\partial z_i} + \frac{\beta}{2} \sum_{j\ne
	i} \frac{1}{z_i-z_j}(1-s_{ij}),\qquad s_{ij}\ \text{permutes}\ z_i\
	\text{and}\ z_j,\end{equation*}
such that 
\begin{equation*}
\mathcal{P}_k B_{\mathbf{a}}(\mathbf{z};\beta) = \Bigl(\sum_{i=1}^N a_i^k\Bigr) B_{\mathbf{a}}(\mathbf{z};\beta),
\qquad k=1,2,\ldots.
\end{equation*}
Alternatively, the functions $B_{\mathbf{a}}(\mathbf{z};\beta)$ arise
as
certain limits of Jack symmetric polynomials, 
see \cite{OkounkovOlshanskiJack1996}, \cite{benaych2022matrix}*{Section~2.2}.
Key properties of the multivariate Bessel functions include:
\begin{enumerate}[$\bullet$]
	\item Normalization $B_{\mathbf{a}}(\mathbf{0};\beta)=1$.\footnote{Here and 
	below, $\mathbf{0}$ stands for the tuple $(0,\ldots,0)$.}
	\item Symmetry in $z_1,\ldots,z_N$.
	\item Extension to an entire function of the $2N$ variables
	$a_1,\ldots,a_N,z_1,\ldots,z_N$ jointly; in particular,
	$B_{\mathbf{a}}(\mathbf{z};\beta)$ is continuous in the label $\mathbf{a}$
	at fixed $\mathbf{z}$ and $\beta$.
	\item Duality between the label and the argument:
	\begin{equation}
		\label{eq:bessel_duality}
		B_{\mathbf{a}}(\mathbf{z};\beta)=B_{\mathbf{z}}(\mathbf{a};\beta).
	\end{equation}
	This is manifest from the expansion of $B_{\mathbf{a}}(\mathbf{z};\beta)$ as the
	${}_{0}F_{0}^{(2/\beta)}$ hypergeometric (Jack) series, which is symmetric in the two
	sets of variables $\mathbf{a}$ and $\mathbf{z}$
	(see~\cite{OkounkovOlshanskiJack1996}*{(4.2)},~\cite{Forrester-LogGas});
	for $\beta=2$ it also follows from the HCIZ integral below.
	\item Positivity: $B_{\mathbf{a}}(\mathbf{z};\beta)>0$ for all real $\mathbf{a}$
	and all $\mathbf{z}$ with pairwise distinct coordinates, which is the only case we use.
	Indeed, by \eqref{eq:bessel_duality} it equals $B_{\mathbf{z}}(\mathbf{a};\beta)$, and the
	corners representation \eqref{eq:Bessel_corners_integral} in
	\Cref{rmk:bessel_as_partition_function} below writes the latter as the integral of a
	positive integrand against a probability density.
	The duality is what makes this representation applicable: it requires the label to have
	pairwise distinct coordinates, and in the formulas below the perturbation $\mathbf{a}$
	may have ties, while $\mathbf{z}$ does not.
\end{enumerate}
For $\beta=2$, the Bessel function admits the HCIZ integral representation:
\begin{equation*}
	B_{\mathbf{a}}(\mathbf{z};2) = \int_{U(N)} \exp\bigl(\mathrm{Tr}(UAU^*Z)\bigr)\, dU,
\end{equation*}
where $A=\operatorname{diag}(a_1,\ldots,a_N)$, $Z=\operatorname{diag}(z_1,\ldots,z_N)$,
and $dU$ is the Haar measure on the unitary group $U(N)$
normalized to have total mass $1$.

\medskip

Given a probability measure $\mu$ on the
ordered $N$-tuples
$\mathbf{b}=(b_1\ge \cdots \ge b_N)\in \R^N$,
define the \emph{Bessel generating function of $\mu$} by
\begin{equation}
	\label{eq:bessel_gf}
	G_N(\mathbf{z};\beta,\mu) \coloneqq \int_{b_1\ge \cdots \ge b_N}
	B_{\mathbf{b}}(\mathbf{z};\beta) \, d\mu(\mathbf{b}).
\end{equation}
By the symmetry and normalization of Bessel functions,
$G_N$ is symmetric in $z_1,\ldots,z_N$ and satisfies $G_N(\mathbf{0};\beta,\mu)=1$.

For $\beta=2$, 
let $M_A = UAU^*$ and $M_B = VBV^*$ be independent random Hermitian matrices
with fixed eigenvalues $\mathbf{a}$ and $\mathbf{b}$, respectively 
(and 
random eigenvectors coming from the 
independent Haar-distributed 
unitary matrices
$U,V\in U(N)$).
Let $\mathbf{c}$ denote the eigenvalues of $M_A + M_B$. Then we have
\begin{equation*}
	\E\bigl[B_{\mathbf{c}}(\mathbf{z};2) \mid \mathbf{a}, \mathbf{b}\bigr]
	= B_{\mathbf{a}}(\mathbf{z};2) \cdot B_{\mathbf{b}}(\mathbf{z};2).
\end{equation*}
For general $\beta \ne 1,2,4$, when no underlying invariant matrix model exists,
the $\beta$-addition $\boxplus_\beta$ is \emph{defined} via the multiplication of
Bessel generating functions~\cite{GorinMarcus2020},~\cite{benaych2022matrix},
see \eqref{eq:beta_addition} below.

\subsection{The G$\beta$E and perturbed density}

The Gaussian $\beta$-ensemble (G$\beta$E) is the symmetric probability measure
on ordered $N$-tuples $\mathbf{g}=(g_{1}\ge \cdots \ge g_{N})\in \R^{N}$ with density
\begin{equation}
	\label{eq:GbetaE_density_unperturbed}
   f(\mathbf{g})\coloneqq \frac{1}{Z_{N,\beta}}\prod_{1\le i<j\le N}(g_{i}-g_{j})^{\beta}\prod_{i=1}^{N}\exp\left(-\frac{1}{2}g_{i}^{2}\right), 
\end{equation}
where $Z_{N,\beta}$ is the probability normalizing constant.
The scaling inside the exponent is chosen so that the Bessel generating function
\eqref{eq:bessel_gf}
of the G$\beta$E is
\begin{equation}
\label{eq:GbetaE_bessel_gf}
G_N(\mathbf{x};\beta,\text{G}\beta\text{E}) =
\exp\Bigl(\frac{1}{2}\sum_{i=1}^{N}x_{i}^{2}\Bigr).
\end{equation}

Consider the $\beta$-addition of
an $N$-dimensional G$\beta$E random vector $\mathbf{g}$
and a deterministic $N$-tuple $\mathbf{a}=(a_{1},\ldots,a_{N})\in \R^{N}$,
denoted by
\begin{equation}
	\label{eq:beta_addition}
	\mathbf{c}=\mathbf{a}\boxplus_{\beta}\mathbf{g}.
\end{equation}
By definition, the Bessel generating function of $\mathbf{c}$
is the product
\begin{equation*}
	\E\bigl[B_{\mathbf{c}}(\mathbf{x};\beta)\bigr]\coloneqq B_{\mathbf{a}}(\mathbf{x};\beta)\cdot \exp\Bigl(\frac{1}{2}\sum_{i=1}^{N}x_{i}^{2}\Bigr).
\end{equation*}
The operation $\boxplus_{\beta}$ 
\eqref{eq:beta_addition}
is the $\beta$-extension
of the matrix addition \eqref{eq:perturbed_GOE_GUE_GSE}
from the classical values $\beta=1,2,4$.

\begin{proposition}
	\label{prop:density_perturbed_GbetaE}
    The probability density of $\mathbf{c}$ 
		\eqref{eq:beta_addition}
		has the form
	\begin{equation}
		\label{eq:density_perturbed_GbetaE}
        f_{\mathbf{a}}(\mathbf{x})\coloneqq\frac{1}{Z_{N,\beta}}\exp\Bigl(-\frac{1}{2}\sum_{i=1}^{N}a_{i}^{2}\Bigr)\exp\Bigl(-\frac{1}{2}\sum_{i=1}^{N}x_{i}^{2}\Bigr)\prod_{1\le i<j\le N}(x_{i}-x_{j})^{\beta}B_{\mathbf{a}}(\mathbf{x};\beta),
    \end{equation}
		where $Z_{N,\beta}$ is the same normalizing constant as in 
		\eqref{eq:GbetaE_density_unperturbed}.
\end{proposition}
\begin{proof}
	The random vector $\mathbf{c}$ is equal in distribution to the marginal at
	time $1$ of the standard $\beta$-Dyson Brownian motion started from
	$\mathbf{a}$.  The density formula \eqref{eq:density_perturbed_GbetaE}
	follows from \cite{gorin2024airy}*{(23)}. For additional references, see the
	proof of Lemma~3.8 in \cite{gorin2024airy}.
\end{proof}

The single-level law \eqref{eq:density_perturbed_GbetaE} is the Gaussian
$\beta$-ensemble with external source: after the rescaling
$\mathbf{x}=\sqrt{\beta}\ssp\mathbf{u}$ and
$\mathbf{a}=\sqrt{\beta}\ssp\mathbf{f}$, it coincides with the density
studied by Desrosiers and Liu~\cite{DesrosiersLiu2015}; see
also~\cite{Forrester2013source}.
Our contribution begins with the compatible multilevel corners law of
\Cref{def:perturbed_beta_corners}.

\subsection{The $\beta$-corners process}

The Gaussian (Hermite) $\beta$-corners process is a probability distribution on interlacing arrays
$\mathbf{y}^1 \prec \mathbf{y}^2 \prec \cdots \prec \mathbf{y}^{N-1} \prec \mathbf{y}^N$
with levels 
\begin{equation*}
	\mathbf{y}^k = (y^k_1 \ge y^k_2 \ge \cdots \ge y^k_k)\in \R^k.
\end{equation*}
For the classical values $\beta = 1, 2, 4$, the corners process describes the joint distribution
of eigenvalues of all principal submatrices of a GOE/GUE/GSE matrix, respectively.

Consider first the unperturbed $\beta$-corners process
(see \cite{GorinShkolnikov2014}; the general $\beta$ density
already appears in \cite{Neretin2003Rayleigh})
with fixed top row
$\mathbf{y}^N = \mathbf{z} = (z_1 > \cdots > z_N)$.\footnote{The 
density requires a strict top row, since the 
normalizing constant includes $\prod_{i<j}(z_i-z_j)^{\beta-1}$.
On the other hand, under the G$\beta$E density, the coordinates of $\mathbf{z}$ are distinct with probability one.}
Denote by
$\mathbf{Y} = (\mathbf{y}^1, \ldots, \mathbf{y}^{N-1})$ the collection of lower levels.
The conditional density of $\mathbf{Y}$ given $\mathbf{z}$ is, by definition,
\begin{equation}
	\label{eq:unperturbed_links}
	\Lambda(\mathbf{Y};\mathbf{z})
	=
	\frac{\mathbf{1}_{\mathbf{y}^{1}\prec \cdots \prec \mathbf{y}^{N-1}\prec \mathbf{z}}}
	{\tilde{Z}_{N,\beta,\mathbf{z}}}
	\prod_{k=1}^{N-1}\Biggl(
		\prod_{1\le i<j\le k}(y^{k}_{i}-y_{j}^{k})^{2-\beta}
		\prod_{p=1}^{k}\prod_{q=1}^{k+1}|y_{p}^{k}-y_{q}^{k+1}|^{\beta/2-1}
	\Biggr)
	.
\end{equation}
The normalizing constant is given by the
Dixon--Anderson formula~\cite{dixon1905generalization},~\cite{Anderson1991Selberg}
(see also~\cite{Forrester-LogGas}*{Chapter~4}):
\begin{equation}
	\label{eq:Dixon_Anderson}
	\tilde{Z}_{N,\beta,\mathbf{z}} = \prod_{k=1}^{N} \frac{\Gamma(\beta/2)^k}{\Gamma(k\beta/2)} \cdot \prod_{1\le i < j \le N} (z_i - z_j)^{\beta-1}.
\end{equation}
The joint density of the full unperturbed Gaussian $\beta$-corners process 
thus has the form
\begin{equation*}
F(\mathbf{Y},\mathbf{z})=
	f(\mathbf{z})\cdot \Lambda(\mathbf{Y};\mathbf{z}),
\end{equation*}
where $f$ is the G$\beta$E density \eqref{eq:GbetaE_density_unperturbed}.

\begin{remark}
\label{rmk:bessel_as_partition_function}
	The $\beta$-corners process provides an integral representation for the multivariate Bessel function~\cite{GuhrKohler2002}:
	\begin{equation}
		\label{eq:Bessel_corners_integral}
		B_{\mathbf{z}}(\mathbf{x};\beta)
		= 
		\int 
		\Lambda(\mathbf{Y};\mathbf{z})\cdot
		\exp\left(
			\sum_{k=1}^{N} x_k \cdot \Bigl(\sum_{i=1}^{k} y_i^k - \sum_{j=1}^{k-1} y_j^{k-1}\Bigr)
		\right)
		d\mathbf{Y},
	\end{equation}
	where the integral is with respect to the $\beta$-corners process 
	conditioned on having the
	fixed top row
	$\mathbf{y}^N = \mathbf{z}$, see
	\eqref{eq:unperturbed_links}.
\end{remark}

\subsection{Perturbed $\beta$-corners process}

Combining the perturbed G$\beta$E \eqref{eq:density_perturbed_GbetaE} with the unperturbed links
\eqref{eq:unperturbed_links}, we now give the main definition:

\begin{definition}
	\label{def:perturbed_beta_corners}
	The \emph{perturbed Gaussian $\beta$-corners process}
	with perturbation $\mathbf{a}=(a_1,\ldots,a_N)$
	is a random interlacing array
	$\mathbf{y}^{1} \prec \mathbf{y}^{2} \prec \cdots \prec \mathbf{y}^{N-1} \prec \mathbf{y}^{N} = \mathbf{z}$,
	with joint density
	\begin{equation}
		\label{eq:perturbed_corners_density}
		F_{\mathbf{a}}(\mathbf{Y},\mathbf{z})\coloneqq
		f_{\mathbf{a}}(\mathbf{z})\cdot \Lambda_{\mathbf{a}}(\mathbf{Y};\mathbf{z}),
	\end{equation}
	where $f_{\mathbf{a}}$ is the perturbed G$\beta$E density \eqref{eq:density_perturbed_GbetaE}, and
	the perturbed links are defined as
	\begin{equation}
		\label{eq:perturbed_corners_conditional}
			\Lambda_{\mathbf{a}}(\mathbf{Y};\mathbf{z})
		 \coloneqq
			\Lambda(\mathbf{Y};\mathbf{z})\cdot
		 \frac{1}{B_{\mathbf{a}}(\mathbf{z};\beta)}
		 \exp\left[a_{N}\sum_{i=1}^{N}z_{i}+\sum_{k=1}^{N-1}\sum_{i=1}^{k}y^{k}_{i}(a_{k}-a_{k+1})\right].
	\end{equation}
\end{definition}

The following statement shows that the perturbed corners process has the desired
Bessel generating function:
\begin{proposition}
	\label{prop:perturbed_Bessel_identity}
	The perturbed Gaussian $\beta$-corners process satisfies,
	for every $\mathbf{x}\in\R^{N}$,
    \begin{equation*}
\int\int
				F_{\mathbf{a}}(\mathbf{Y},\mathbf{z})
				\exp\Bigl[x_{N}\sum_{i=1}^{N}z_{i}+\sum_{k=1}^{N-1}\sum_{i=1}^{k}y^{k}_{i}(x_{k}-x_{k+1})\Bigr]
				\, d\mathbf{Y}\,d\mathbf{z}
				=
				\exp\Bigl(\frac{1}{2}\sum_{i=1}^{N}x_{i}^{2}\Bigr)\exp\Bigl(\sum_{i=1}^{N}a_{i}x_{i}\Bigr),
    \end{equation*}
		where the integration is over all interlacing arrays
		$\mathbf{y}^{1} \prec \mathbf{y}^{2} \prec \cdots \prec \mathbf{y}^{N-1} \prec \mathbf{z}$,
		and we denote $\mathbf{Y}=(\mathbf{y}^{1},\ldots,\mathbf{y}^{N-1})$.
\end{proposition}
\begin{proof}
    By \eqref{eq:Bessel_corners_integral} and the duality \eqref{eq:bessel_duality}
    (which lets us write $B_{\mathbf{z}}(\mathbf{a};\beta)=B_{\mathbf{a}}(\mathbf{z};\beta)$),
		we have
    \begin{equation}
		\label{eq:perturbed_Bessel_identity_proof_first_display}
		\int \Lambda(\mathbf{Y};\mathbf{z})\exp\Bigl[\sum_{k=1}^{N-1}\sum_{i=1}^{k}y^{k}_{i}(a_{k}-a_{k+1})\Bigr]
		\ssp d\mathbf{Y}
		=
		B_{\mathbf{a}}(\mathbf{z};\beta)\exp\Bigl(-a_{N}\sum_{i=1}^{N}z_{i}\Bigr).
	\end{equation}
	Therefore,
	\begin{align*}
        &\int \Lambda_{\mathbf{a}}(\mathbf{Y};\mathbf{z})\exp\Bigl[\sum_{k=1}^{N-1}\sum_{i=1}^{k}y^{k}_{i}(x_{k}-x_{k+1})\Bigr]\,d\mathbf{Y}\\
        &\hspace{20pt}=\frac{1}{B_{\mathbf{a}}(\mathbf{z};\beta)}\exp\Bigl(a_{N}\sum_{i=1}^{N}z_{i}\Bigr)B_{\mathbf{a}+\mathbf{x}}(\mathbf{z};\beta)\exp\Bigl(-(a_{N}+x_{N})\sum_{i=1}^{N}z_{i}\Bigr)\\
        &\hspace{20pt}=\frac{B_{\mathbf{a}+\mathbf{x}}(\mathbf{z};\beta)}{B_{\mathbf{a}}(\mathbf{z};\beta)}\exp\Bigl(-x_{N}\sum_{i=1}^{N}z_{i}\Bigr).
    \end{align*}
		Let us multiply by the remaining $\mathbf{z}$-dependent factors and integrate over $\mathbf{z}$.
		We have
    \begin{align*}
        &\int \frac{B_{\mathbf{a}+\mathbf{x}}(\mathbf{z};\beta)}{B_{\mathbf{a}}(\mathbf{z};\beta)}f_{\mathbf{a}}(\mathbf{z})d\mathbf{z}\\
		&\hspace{20pt}=\int \frac{1}{Z_{N,\beta}}\exp\Bigl(-\frac{1}{2}\sum_{i=1}^{N}a_{i}^{2}\Bigr)\exp\Bigl(-\frac{1}{2}\sum_{i=1}^{N}z_{i}^{2}\Bigr)\prod_{1\le i<j\le N}(z_{i}-z_{j})^{\beta}B_{\mathbf{a}+\mathbf{x}}(\mathbf{z};\beta)d\mathbf{z}\\
        &\hspace{20pt}=\exp\Bigl(-\frac{1}{2}\sum_{i=1}^{N}a_{i}^{2}\Bigr)\exp\Bigl(\frac{1}{2}\sum_{i=1}^{N}(a_{i}+x_{i})^{2}\Bigr)
		=\exp\Bigl(\frac{1}{2}\sum_{i=1}^{N}x_{i}^{2}\Bigr)\exp\Bigl(\sum_{i=1}^{N}a_{i}x_{i}\Bigr),
    \end{align*}
		where the last line follows from the duality \eqref{eq:bessel_duality}
		and the Bessel generating function property~\eqref{eq:GbetaE_bessel_gf}
		of the unperturbed G$\beta$E.
	This completes the proof.
\end{proof}

\subsection{Matching with matrix models}
\label{sec:match_matrix_models}

Let us verify that 
\Cref{def:perturbed_beta_corners}
reduces to the matrix model construction of \Cref{sec:beta124}
when $\beta=1,2$, or $4$.

Expanding the joint density \eqref{eq:perturbed_corners_density}, we have
\begin{equation}
\label{eq:perturbed_corners_density_explicit}
\begin{split}
		F_{\mathbf{a}}(\mathbf{Y},\mathbf{z})&=\frac{1}{Z_{N,\beta}}\ssp
		\Lambda(\mathbf{Y};\mathbf{z})
		\exp\Bigl[a_{N}\sum_{i=1}^{N}z_{i}+\sum_{k=1}^{N-1}\sum_{i=1}^{k}y^{k}_{i}(a_{k}-a_{k+1})\Bigr]
		\\&\hspace{80pt}\times
		\exp\Bigl(-\frac{1}{2}\sum_{i=1}^{N}(a_{i}^{2}+z_{i}^{2})\Bigr)
		\prod_{1\le i<j\le N}(z_{i}-z_{j})^{\beta}.
	\end{split}
\end{equation}
Since $\Lambda(\mathbf{Y};\mathbf{z})$ 
\eqref{eq:unperturbed_links}--\eqref{eq:Dixon_Anderson}
contains the $\mathbf{z}$-dependent factor
\begin{equation*}
	\prod_{1\le i<j\le N}(z_{i}-z_{j})^{1-\beta}\prod_{p=1}^{N-1}\prod_{q=1}^{N}|y^{N-1}_{p}-z_{q}|^{\beta/2-1}\mathbf{1}_{\mathbf{z}\succ \mathbf{y}^{N-1}},
\end{equation*}
the conditional density of $\mathbf{z}$ given the lower levels is
\begin{equation}\label{eq_transitionkernel}
	h(\mathbf{z}\mid \mathbf{Y})
	\propto
	\exp\Bigl(-\frac{1}{2}\sum_{i=1}^{N}z_{i}^{2}\Bigr)
	\exp\Bigl(a_{N}\sum_{i=1}^{N}z_{i}\Bigr)
	\prod_{1\le i<j\le N}(z_{i}-z_{j})
	\prod_{p=1}^{N-1}\prod_{q=1}^{N}|y^{N-1}_{p}-z_{q}|^{\beta/2-1}
	\mathbf{1}_{\mathbf{z}\succ \mathbf{y}^{N-1}}.
\end{equation}
This depends on $\mathbf{Y}$ only through $\mathbf{y}^{N-1}$, and it matches
the conditional density \eqref{eq_cond_density_beta124}
from \Cref{prop:cond_density_beta124}.
That identifies the top transition. It remains to check that the levels below
the top form a perturbed corners process of size $N-1$.

\begin{lemma}
	\label{lem:corners_consistency}
	Let $N\ge2$ and write $\mathbf{a}^{(k)}=(a_{1},\ldots,a_{k})$.
	\begin{enumerate}[\normalfont(i)]
		\item For $\beta>0$ and $1\le k\le N-1$, the conditional density of
		$(\mathbf{y}^{1},\ldots,\mathbf{y}^{k})$ given
		$(\mathbf{y}^{k+1},\ldots,\mathbf{y}^{N})$ under
		\Cref{def:perturbed_beta_corners} is
		$\Lambda_{\mathbf{a}^{(k+1)}}(\ \cdot\ ;\mathbf{y}^{k+1})$.
		\item For $\beta=1,2,4$ the marginal density of $\mathbf{y}^{N-1}$ is
		$f_{\mathbf{a}^{(N-1)}}$, and $(\mathbf{y}^{1},\ldots,\mathbf{y}^{N-1})$
		is the perturbed $\beta$-corners process with perturbation
		$\mathbf{a}^{(N-1)}$.
	\end{enumerate}
\end{lemma}

\begin{proof}
	(i) Fix $\mathbf{y}^{k+1},\ldots,\mathbf{y}^{N}$. In
	\eqref{eq:perturbed_corners_density}--\eqref{eq:perturbed_corners_conditional}
	the factors $f_{\mathbf{a}}(\mathbf{z})$, $B_{\mathbf{a}}(\mathbf{z};\beta)$,
	$\tilde{Z}_{N,\beta,\mathbf{z}}$ and $\exp(a_{N}\sum_{i}z_{i})$, together with
	the factors of \eqref{eq:unperturbed_links} indexed by $k'\ge k+1$, do not
	involve $\mathbf{y}^{1},\ldots,\mathbf{y}^{k}$, and the interlacing indicator
	splits as
	$\mathbf{1}_{\mathbf{y}^{1}\prec\cdots\prec\mathbf{y}^{k+1}}\ssp
	\mathbf{1}_{\mathbf{y}^{k+1}\prec\cdots\prec\mathbf{z}}$.
	What remains is
	\begin{equation*}
		\mathbf{1}_{\mathbf{y}^{1}\prec\cdots\prec\mathbf{y}^{k+1}}
		\prod_{k'=1}^{k}\Biggl(
			\prod_{1\le i<j\le k'}(y^{k'}_{i}-y_{j}^{k'})^{2-\beta}
			\prod_{p=1}^{k'}\prod_{q=1}^{k'+1}\bigl|y_{p}^{k'}-y_{q}^{k'+1}\bigr|^{\beta/2-1}
		\Biggr)
		\exp\Biggl[\ssp\sum_{k'=1}^{k}\sum_{i=1}^{k'}y^{k'}_{i}(a_{k'}-a_{k'+1})\Biggr],
	\end{equation*}
	which is $\Lambda_{\mathbf{a}^{(k+1)}}(\ \cdot\ ;\mathbf{y}^{k+1})$ times a
	factor depending on $\mathbf{y}^{k+1}$ alone. Both are probability densities
	in $(\mathbf{y}^{1},\ldots,\mathbf{y}^{k})$, so they coincide.

	(ii) Write $\boldsymbol{\mu}=\mathbf{y}^{N-1}$ and
	$\mathbf{Y}'=(\mathbf{y}^{1},\ldots,\mathbf{y}^{N-2})$. Splitting off the
	$k=N-1$ factor of \eqref{eq:unperturbed_links} and integrating out
	$\mathbf{Y}'$ by the corners representation
	\eqref{eq:Bessel_corners_integral} at size $N-1$ with
	$\mathbf{x}=\mathbf{a}^{(N-1)}$ gives
	\begin{equation*}
		\int \Lambda(\mathbf{Y}';\boldsymbol{\mu})
		\exp\Biggl[\ssp\sum_{k=1}^{N-2}\sum_{i=1}^{k}y^{k}_{i}(a_{k}-a_{k+1})\Biggr]
		\ssp d\mathbf{Y}'
		=\exp\Bigl(-a_{N-1}\sum_{i=1}^{N-1}\mu_{i}\Bigr)
		\ssp B_{\boldsymbol{\mu}}(\mathbf{a}^{(N-1)};\beta).
	\end{equation*}
	Combining the two Vandermonde powers by \eqref{eq:Dixon_Anderson}, the joint
	density of $(\boldsymbol{\mu},\mathbf{z})$ is a constant multiple of
	\begin{equation*}
		\prod_{1\le i<j\le N-1}(\mu_{i}-\mu_{j})\ssp
		\exp\Bigl(-a_{N}\sum_{i=1}^{N-1}\mu_{i}\Bigr)
		B_{\boldsymbol{\mu}}(\mathbf{a}^{(N-1)};\beta)\ssp
		\mathcal{K}(\boldsymbol{\mu},\mathbf{z}),
	\end{equation*}
	where $\mathcal{K}$ is the right-hand side of \eqref{eq_cond_density_beta124}
	with $\boldsymbol{\lambda}=\mathbf{z}$. Since the density of
	\Cref{rmk:collecting_mu_terms_for_full_density} integrates to~$1$,
	\begin{equation*}
		\int \mathcal{K}(\boldsymbol{\mu},\mathbf{z})\ssp d\mathbf{z}
		=\Gamma(\beta/2)^{N-1}\sqrt{2\pi}\ssp e^{a_{N}^{2}/2}
		\prod_{1\le i<j\le N-1}(\mu_{i}-\mu_{j})^{\beta-1}
		\exp\Bigl(a_{N}\sum_{i=1}^{N-1}\mu_{i}
		-\frac{1}{2}\sum_{i=1}^{N-1}\mu_{i}^{2}\Bigr).
	\end{equation*}
	The two exponentials in $a_{N}$ cancel and the Vandermonde powers add up
	to~$\beta$, so the marginal density of $\boldsymbol{\mu}$ is a constant
	multiple of
	$B_{\boldsymbol{\mu}}(\mathbf{a}^{(N-1)};\beta)
	\exp\bigl(-\frac{1}{2}\sum_{i}\mu_{i}^{2}\bigr)
	\prod_{i<j}(\mu_{i}-\mu_{j})^{\beta}$,
	which by the duality \eqref{eq:bessel_duality} is
	\eqref{eq:density_perturbed_GbetaE} at size $N-1$. Together with~(i) for
	$k=N-2$, this is the second assertion.
\end{proof}

Both constructions now obey the same recursion in~$N$. On the one hand,
\eqref{eq_transitionkernel} and \Cref{lem:corners_consistency} give
\begin{equation}
	\label{eq:corners_recursion}
	F_{\mathbf{a}}(\mathbf{Y},\mathbf{z})
	=F_{\mathbf{a}^{(N-1)}}\bigl(\mathbf{Y}',\mathbf{y}^{N-1}\bigr)\ssp
	h(\mathbf{z}\mid \mathbf{y}^{N-1}),
	\qquad
	\mathbf{Y}'=(\mathbf{y}^{1},\ldots,\mathbf{y}^{N-2}).
\end{equation}
On the other hand, the corner $C_{N-1}$ of $C=A+G$
\eqref{eq:perturbed_GOE_GUE_GSE} is itself a perturbed GOE/GUE/GSE matrix with
perturbation $\mathbf{a}^{(N-1)}$, so the eigenvalues of $C_{1},\ldots,C_{N-1}$
form the corresponding array one size down. Conditionally on the whole of
$C_{N-1}$, the top row $\boldsymbol{\lambda}$ has the density
\eqref{eq_cond_density_beta124}, since $\boldsymbol{\lambda}$ is determined by
$\boldsymbol{\mu}$ and $\boldsymbol{\xi}$ through \eqref{eq_polynomialeqn},
and $\boldsymbol{\xi}$ is independent of $C_{N-1}$ by the discussion
after \eqref{eq_Pi}. So the matrix model array satisfies
\eqref{eq:corners_recursion} as well. For $N=1$ both laws are
$\mathcal{N}(a_{1},1)$. By induction on~$N$, for $\beta=1,2,4$ the perturbed
$\beta$-corners process of \Cref{def:perturbed_beta_corners}, defined through
multivariate Bessel functions, agrees with the perturbed corners
processes constructed from the matrix models.

\subsection{Down transition density}
\label{sec:down_transition}

The marginal and conditional densities also give the
\emph{down transition density}
$\Lambda_\mathbf{a}(\boldsymbol{\mu}, \boldsymbol{\lambda})$,
which goes in the opposite direction from $h(\boldsymbol{\lambda} \mid \boldsymbol{\mu})$:

\begin{proposition}
	\label{prop:cond_mu_given_lambda}
	Let $N\ge2$.
	The conditional density of
	$\mathbf{y}^{N-1}=\boldsymbol{\mu}=(\mu_1 \ge \cdots \ge \mu_{N-1})$
	given $\mathbf{y}^N = \boldsymbol{\lambda} = (\lambda_1 \ge \cdots \ge \lambda_N)$
	under the perturbed $\beta$-corners process
	is proportional to
	\begin{equation}
		\label{eq:cond_mu_given_lambda}
		\Lambda_\mathbf{a}(\boldsymbol{\mu}, \boldsymbol{\lambda})
		\propto
		\prod_{1\le i<j\le N-1}(\mu_i - \mu_j)
		\prod_{i=1}^{N-1}\prod_{j=1}^{N}|\mu_i - \lambda_j|^{\beta/2-1}
		B_{\mathbf{a}'}(\boldsymbol{\mu}; \beta)
		\exp\Bigl(-a_N \sum_{i=1}^{N-1}\mu_i\Bigr)
		\mathbf{1}_{\boldsymbol{\mu} \prec \boldsymbol{\lambda}},
	\end{equation}
	where $\mathbf{a}' = (a_1, \ldots, a_{N-1})$
	and the proportionality constant depends on $\boldsymbol{\lambda}$ and $\mathbf{a}$.
\end{proposition}
\begin{proof}
	The conditional density of $\boldsymbol{\mu}=\mathbf{y}^{N-1}$ given
	$\boldsymbol{\lambda}=\mathbf{y}^{N}=\mathbf{z}$ is proportional, as a function of
	$\boldsymbol{\mu}$, to the joint density of $(\mathbf{y}^{N-1},\mathbf{y}^{N})$
	obtained from \eqref{eq:perturbed_corners_density} by integrating out the lower
	levels $\mathbf{y}^{1},\ldots,\mathbf{y}^{N-2}$.
	Peeling off the $k=N-1$ factor from the links \eqref{eq:unperturbed_links} and
	writing $\Lambda'(\mathbf{Y}';\boldsymbol{\mu})$ for the $(N-1)$-level unperturbed
	links with top row $\boldsymbol{\mu}$, where $\mathbf{Y}'=(\mathbf{y}^{1},\ldots,\mathbf{y}^{N-2})$,
	we have
	\begin{equation*}
		\Lambda(\mathbf{Y};\mathbf{z})
		=\frac{\tilde{Z}_{N-1,\beta,\boldsymbol{\mu}}}{\tilde{Z}_{N,\beta,\mathbf{z}}}\,
		\Lambda'(\mathbf{Y}';\boldsymbol{\mu})
		\prod_{1\le i<j\le N-1}(\mu_i-\mu_j)^{2-\beta}
		\prod_{p=1}^{N-1}\prod_{q=1}^{N}|\mu_p-\lambda_q|^{\beta/2-1}
		\mathbf{1}_{\boldsymbol{\mu}\prec\boldsymbol{\lambda}}.
	\end{equation*}
	The exponential in \eqref{eq:perturbed_corners_conditional}
	splits as $\sum_{k=1}^{N-2}\sum_{i=1}^{k} y^{k}_{i}(a_k-a_{k+1})
	+(a_{N-1}-a_N)\sum_{i=1}^{N-1}\mu_i$.
	Integrating the lower levels against $\Lambda'(\mathbf{Y}';\boldsymbol{\mu})$ by
	\eqref{eq:Bessel_corners_integral} in dimension
	$N-1$ (exactly as in the proof of \Cref{prop:perturbed_Bessel_identity}) gives
	\begin{equation*}
		\int \Lambda'(\mathbf{Y}';\boldsymbol{\mu})
		\exp\Bigl[\sum_{k=1}^{N-2}\sum_{i=1}^{k} y^{k}_{i}(a_k-a_{k+1})\Bigr]
		\,d\mathbf{Y}'
		=B_{\mathbf{a}'}(\boldsymbol{\mu};\beta)\exp\Bigl(-a_{N-1}\sum_{i=1}^{N-1}\mu_i\Bigr).
	\end{equation*}
	Here $f_\mathbf{a}(\mathbf{z})$, $B_\mathbf{a}(\mathbf{z};\beta)$,
	$\tilde{Z}_{N,\beta,\mathbf{z}}$ and $\exp(a_N\sum_i z_i)$ are constants in
	$\boldsymbol{\mu}$.  Using $$\tilde{Z}_{N-1,\beta,\boldsymbol{\mu}}\propto
	\prod_{1\le i<j\le N-1}(\mu_i-\mu_j)^{\beta-1}$$ from \eqref{eq:Dixon_Anderson},
	the two Vandermonde powers combine as
	$(\mu_i-\mu_j)^{\beta-1}(\mu_i-\mu_j)^{2-\beta}=(\mu_i-\mu_j)$, and the two
	exponentials as $(a_{N-1}-a_N)\sum_i\mu_i-a_{N-1}\sum_i\mu_i=-a_N\sum_i\mu_i$.
	This completes the proof of \eqref{eq:cond_mu_given_lambda}.
\end{proof}

\subsection{Random polynomial equation}
\label{sec:random_poly_general}

The polynomial equation~\eqref{eq_polynomialeqn} derived in \Cref{sec:beta124}
for $\beta=1,2,4$ extends to general $\beta>0$.
This provides an alternative characterization of the perturbed $\beta$-corners process,
which we use in the zero-temperature analysis below.

\begin{proposition}[Random polynomial equation for general $\beta$]
\label{prop:random_poly_general}
Let $\boldsymbol{\lambda}=(\lambda_1\ge\cdots\ge\lambda_N)$ and
$\boldsymbol{\mu}=(\mu_1\ge\cdots\ge\mu_{N-1})$ be consecutive levels
of the perturbed $\beta$-corners process with perturbation $\mathbf{a}=(a_1,\ldots,a_N)$.
Conditionally on $\boldsymbol{\mu}$, the eigenvalues $\boldsymbol{\lambda}$
are the $N$ roots of the random equation
\begin{equation}
	\label{eq:random_poly_general}
	\sum_{i=1}^{N-1}\frac{\xi_i}{x-\mu_i} = x - \xi_N,
\end{equation}
where the random variables $\xi_1,\ldots,\xi_{N-1},\xi_N$ are conditionally independent given $\boldsymbol{\mu}$, with
\begin{equation}
	\label{eq:xi_distributions}
	\xi_i \sim \frac{1}{2}\chi^2_\beta , \quad i=1,\ldots,N-1,
	\qquad\text{and}\qquad
	\xi_N \sim \mathcal{N}(a_N,1).
\end{equation}
Equivalently, the eigenvalues satisfy the polynomial identity
\begin{equation}
	\label{eq:poly_identity_general}
	\prod_{i=1}^{N-1}(x-\mu_i)(x-\xi_N) - \prod_{j=1}^{N}(x-\lambda_j)
	= \sum_{i=1}^{N-1}\xi_i\prod_{k\ne i}(x-\mu_k).
\end{equation}
\end{proposition}

\begin{proof}
The computation repeats the one in \Cref{sec:beta124} word for word.
There, the classical value of $\beta$ is used only to identify the law of
$\xi_i=\|P_i\|^{2}$ as $\tfrac{1}{2}\chi^{2}_{\beta}$;
the polynomial equation~\eqref{eq_polynomialeqn}, the Jacobian
of \Cref{sec:Jacobian}, and the assembly of the conditional density in the
proof of \Cref{prop:cond_density_beta124} are algebraic identities carrying
$\beta$ only through the exponent $\beta/2-1$.
For general $\beta>0$ the distributions~\eqref{eq:xi_distributions} are
postulated rather than read off a matrix model, and the same steps turn the
density of $\boldsymbol{\xi}$ into the transition
kernel~\eqref{eq_transitionkernel}.
\end{proof}

Let us make two concluding remarks about the
perturbed $\beta$-corners process before turning to 
asymptotics.

\begin{remark}
	Assume $N\ge2$.
	The interlacing $\boldsymbol{\mu}\prec\boldsymbol{\lambda}$ follows
	from~\eqref{eq:random_poly_general}: the left-hand side
	has poles at each $\mu_i$ with positive residues $\xi_i>0$,
	while the right-hand side is linear in $x$.
	By analyzing signs, exactly one root $\lambda_j$ lies in each interval
	$(\mu_{i+1},\mu_i)$, $i=1,\dots,N-2$, plus one root above $\mu_1$ and one
	below $\mu_{N-1}$.
	(For $N=1$ equation~\eqref{eq:random_poly_general} reduces to
	$\lambda_1=\xi_N$, with no lower row.)
\end{remark}

\begin{remark}[Relation to finite free probability]
	\label{rmk:finite_free_prob}
	The polynomial identity~\eqref{eq:poly_identity_general} connects to
	finite free probability~\cite{MarcusSpielmanSrivastava2022},~\cite{Marcus2021},~\cite{GorinMarcus2020}.
	Comparing the coefficients of $x^{N-1}$ and of $x^{N-2}$
	in~\eqref{eq:poly_identity_general} returns the
	identities~\eqref{eq_xi2} and~\eqref{eq_xi3}, which thus hold for all $\beta>0$.
	These moment relations encode how the spectrum of a sum of matrices
	relates to the spectra of the summands.
\end{remark}

\section{Zero-temperature asymptotics: fixed perturbation}
\label{sec:beta_inf_trivial}

In this section and the next one we analyze the $\beta\to\infty$ asymptotics
of the perturbed $\beta$-corners process.
Here, the perturbation parameters $\mathbf{a}=(a_1,\ldots,a_N)$ remain fixed
as $\beta$ grows; a perturbation growing with $\beta$
is considered in \Cref{sec:beta_inf_scaled}.

\subsection{The $\infty$-corners lattice}

The $\beta\to\infty$ behavior of the \emph{unperturbed} $\beta$-corners
process~\eqref{eq:unperturbed_links} was determined in~\cite{GorinMarcus2020}: 
the array freezes onto a deterministic lattice
of polynomial-derivative roots, with Gaussian fluctuations of order
$\beta^{-1/2}$ on top of it. In this subsection, 
we recall the results from \cite{GorinMarcus2020}
in the form useful to us.

\begin{definition}
\label{def:deterministic_lattice}
Given a top row $\mathbf{z}=(z_1>\cdots>z_N)$, the \emph{$\infty$-corners lattice}
$\mathsf{Y}\coloneqq\{\mathsf{y}_i^k\}_{1\le i\le k\le N}$ is defined as follows:
$\mathsf{y}_i^N = z_i$ for $i=1,\ldots,N$, and for $k=N-1,\ldots,1$,
the points $\mathsf{y}_1^k > \cdots > \mathsf{y}_k^k$ are the $k$ roots of the polynomial
\begin{equation*}
Q^{N\to k}(u) = \frac{1}{N(N-1)\cdots(k+1)}
	\left(\frac{d}{du}\right)^{N-k}\prod_{j=1}^N(u-z_j).
\end{equation*}
Equivalently~\cite{GorinMarcus2020}*{Section~3.4}, the lattice positions
satisfy the \emph{optimality equations}
\begin{equation}
	\label{eq:optimality_equations}
	\sum_{j=1}^{k+1}\frac{1}{\mathsf{y}_i^k - \mathsf{y}_j^{k+1}} = 0,
	\qquad i=1,\ldots,k,
\end{equation}
which characterize the critical point of the functional
$\prod_{p=1}^{k}\prod_{q=1}^{k+1}|u_p - \mathsf{y}_q^{k+1}|$
in the free variable $\mathbf{u}=(u_1>\cdots>u_k)$, the row above being pinned
to the lattice. Here $\mathbf{u}$ ranges over the interlacing chamber
$\mathbf{u}\prec\mathsf{y}^{k+1}$, which is bounded and on whose boundary the
functional vanishes, so the critical point is an interior maximum. It is unique:
the logarithm of the functional splits as
$\sum_{p}\sum_{q}\log|u_p-\mathsf{y}_q^{k+1}|$, strictly concave in each $u_p$.
\end{definition}

\begin{definition}[Discrete Gaussian Free Field, \cite{GorinMarcus2020}*{Definition~1.5}]
\label{def:dGFF}
The \emph{discrete Gaussian Free Field (dGFF)} on the $\infty$-corners lattice
$\{\mathsf{y}_i^k\}$ is a centered Gaussian vector $\{\zeta_i^k\}_{1\le i\le k\le N}$
with $\zeta_i^N=0$ for all $i=1,\ldots,N$ (the top row is pinned), and with the
remaining variables having joint density proportional to
\begin{equation}
	\label{eq:dGFF_density}
	\exp\left(\sum_{k=1}^{N-1}\left[
		\sum_{1\le i<j\le k}\frac{(\zeta_i^k-\zeta_j^k)^2}{2(\mathsf{y}_i^k-\mathsf{y}_j^k)^2}
		- \sum_{p=1}^k\sum_{q=1}^{k+1}\frac{(\zeta_p^k-\zeta_q^{k+1})^2}{4(\mathsf{y}_p^k-\mathsf{y}_q^{k+1})^2}
	\right]\right).
\end{equation}
\end{definition}

\begin{remark}
	\label{rmk:dGFF_proper}
	The same-level terms in~\eqref{eq:dGFF_density} carry a positive sign, so
	integrability is not automatic. It is proved
	in~\cite{GorinMarcus2020}*{Section~3.4} by integrating the levels one at a
	time from the bottom up, each step being an instance of the Gaussian
	identity~\cite{GorinMarcus2020}*{(64)}.
	In particular, this computation implies that the quadratic form in the exponent
	of~\eqref{eq:dGFF_density} is \emph{negative definite}, as a nonnegative eigenvalue
	of the form
	would make the integral diverge.
\end{remark}

\begin{theorem}[\cite{GorinMarcus2020}*{Theorem~1.6}]
\label{thm:GM_crystallization}
Let $\{x_i^k(\beta)\}_{1\le i\le k\le N}$ be the unperturbed $\beta$-corners
process~\eqref{eq:unperturbed_links} with fixed top row $\mathbf{z}=(z_1>\cdots>z_N)$,
let $\{\mathsf{y}_i^k\}$ be the $\infty$-corners lattice with the same top row,
and let $\{\zeta_i^k\}$ be the dGFF on it. Then, as $\beta\to\infty$, the
following hold jointly in all $1\le i\le k\le N$:
\begin{enumerate}[\normalfont(i)]
	\item $x_i^k(\beta)\to \mathsf{y}_i^k$ in probability;
	\item the array $\bigl\{\sqrt{\beta}\bigl(x_i^k(\beta)-\mathsf{y}_i^k\bigr)\bigr\}$
	converges in distribution to $\{\zeta_i^k\}$.
\end{enumerate}
\end{theorem}

Finally, we record identities tying consecutive levels of the
$\infty$-corners 
lattice together. They are repeatedly used in 
\Cref{sec:beta_inf_scaled}
below.

\begin{lemma}
\label{lem:polynomial_identities}
For $2\le k\le N$ and $1\le i\le k$,
\begin{align}
	\sum_{j=1}^k\prod_{m\ne j}(u-\mathsf{y}_m^k)
		&= k\prod_{q=1}^{k-1}(u-\mathsf{y}_q^{k-1}), \label{eq:polynomial_identity1}\\
	\prod_{j\ne i}(\mathsf{y}_i^k - \mathsf{y}_j^k)
		&= k\prod_{q=1}^{k-1}(\mathsf{y}_i^k - \mathsf{y}_q^{k-1}), \label{eq:polynomial_identity2}
\end{align}
and consequently,
\begin{equation}
	\label{eq:polynomial_identity3}
	\sum_{j\ne i}\frac{1}{\mathsf{y}_i^k - \mathsf{y}_j^k}
	= \frac{1}{2}\sum_{q=1}^{k-1}\frac{1}{\mathsf{y}_i^k - \mathsf{y}_q^{k-1}}.
\end{equation}
\end{lemma}
\begin{proof}
By \Cref{def:deterministic_lattice}, the points $\mathsf{y}_q^{k-1}$ are the roots
of the derivative of $\prod_{j=1}^k(u-\mathsf{y}_j^k)$, which
is~\eqref{eq:polynomial_identity1}; setting $u=\mathsf{y}_i^k$ there
gives~\eqref{eq:polynomial_identity2}.
Differentiating~\eqref{eq:polynomial_identity1} in $u$ and then setting
$u=\mathsf{y}_i^k$ gives
\begin{equation*}
	2\sum_{j\ne i}\prod_{m\ne i,j}(\mathsf{y}_i^k - \mathsf{y}_m^k)
	= k\sum_{q=1}^{k-1}\prod_{r\ne q}(\mathsf{y}_i^k - \mathsf{y}_r^{k-1}),
\end{equation*}
and dividing this by~\eqref{eq:polynomial_identity2}
yields~\eqref{eq:polynomial_identity3}.
\end{proof}

\subsection{Crystallization for a fixed perturbation}

We now show that a perturbation which does not grow with
$\beta$ leaves the zero-temperature limit unchanged.
This needs no new asymptotic analysis. With the top row
fixed, the perturbation enters the law of the array only
through the factor
\begin{equation*}
	\exp\Bigl[\ssp\sum_{k=1}^{N-1}(a_{k}-a_{k+1})\sum_{i=1}^{k}y^{k}_{i}\Bigr]
\end{equation*}
of~\eqref{eq:perturbed_corners_conditional}, which does not involve $\beta$; the
remaining perturbation-dependent factors there do not involve $\mathbf{Y}$ and
cancel upon normalization.

\begin{theorem}[Crystallization for fixed perturbation]
\label{thm:crystallization_fixed}
Let $\{y_i^k(\beta)\}_{1\le i\le k\le N}$ be the perturbed $\beta$-corners process
with fixed perturbation $\mathbf{a}=(a_1,\ldots,a_N)$ and fixed top row
$\mathbf{z}=(z_1>\cdots>z_N)$. Let $\{\mathsf{y}_i^k\}$ be the deterministic
$\infty$-corners lattice, and let $\{\zeta_i^k\}$ be the dGFF on this lattice.
Then as $\beta\to\infty$:
\begin{enumerate}[\normalfont(i)]
	\item $y_i^k(\beta)\to \mathsf{y}_i^k$ in probability;
	\item the array $\bigl\{\sqrt{\beta}\,\bigl(y_i^k(\beta)-\mathsf{y}_i^k\bigr)\bigr\}$
	converges in distribution to $\{\zeta_i^k\}$.
\end{enumerate}
In particular, the limit is independent of the perturbation $\mathbf{a}$.
\end{theorem}

\begin{proof}
The top row is pinned: $y_i^N=z_i=\mathsf{y}_i^N$ and $\zeta_i^N=0$, so the case
$k=N$ is trivial. Let us consider the random interlacing array
$\mathbf{Y}=(\mathbf{y}^1,\ldots,\mathbf{y}^{N-1})$.
Set
\begin{equation}
	\label{eq:tilt}
	h(\mathbf{Y})\coloneqq
	\exp\left[\sum_{k=1}^{N-1}\sum_{i=1}^{k}y_i^k(a_k-a_{k+1})\right].
\end{equation}
In the perturbed
density~\eqref{eq:perturbed_corners_density}--\eqref{eq:perturbed_corners_conditional}
the three factors $f_{\mathbf{a}}(\mathbf{z})$, $B_{\mathbf{a}}(\mathbf{z};\beta)^{-1}$
and $\exp\bigl(a_N\sum_i z_i\bigr)$ do not depend on $\mathbf{Y}$. Hence, at the
fixed top row $\mathbf{z}$, the perturbed process is the unperturbed
$\beta$-corners process~\eqref{eq:unperturbed_links} reweighted by $h$: writing
$\E_{\mathbf{a}}$ and $\E$ for the expectations under the perturbed and the
unperturbed law with top row $\mathbf{z}$, we have
\begin{equation}
	\label{eq:tilt_expectation}
	\E_{\mathbf{a}}\bigl[G(\mathbf{Y})\bigr]
	=\frac{\E\bigl[G(\mathbf{Y})\ssp h(\mathbf{Y})\bigr]}{\E\bigl[h(\mathbf{Y})\bigr]}
	\qquad\text{for every bounded measurable }G.
\end{equation}
Here both densities integrate to $1$ in $\mathbf{Y}$: for
$\Lambda(\,\cdot\,;\mathbf{z})$ this is the Dixon--Anderson
evaluation~\eqref{eq:Dixon_Anderson}, and for
$\Lambda_{\mathbf{a}}(\,\cdot\,;\mathbf{z})$ it follows from 
\eqref{eq:perturbed_Bessel_identity_proof_first_display}.

Note that $h$ does not depend on $\beta$.
Moreover, interlacing confines $\mathbf{Y}$ to the bounded Gelfand--Tsetlin polytope
$\{\mathbf{y}^1\prec\cdots\prec\mathbf{y}^{N-1}\prec\mathbf{z}\}$, on which every
coordinate obeys $z_N\le y_i^k\le z_1$, so that
\begin{equation}
	\label{eq:tilt_bounds}
	\bigl|\log h(\mathbf{Y})\bigr|
	\le N^2\max_{1\le k\le N-1}|a_k-a_{k+1}|\cdot\max_{1\le i\le N}|z_i|
	\eqqcolon C_{\mathbf{a},\mathbf{z}}
	\qquad\text{for every }\beta>0 .
\end{equation}
Thus $h$ is bounded above and below by positive constants depending only on
$\mathbf{a}$ and $\mathbf{z}$. (For $N=1$ the array $\mathbf{Y}$ is empty and
$h\equiv1$, so both assertions are vacuous; we take $N\ge2$ below.)

\emph{(i)}
Let $G$ be bounded and continuous. By part~(i) of \Cref{thm:GM_crystallization},
$\mathbf{Y}\to\mathsf{Y}$ in probability under the unperturbed law; since $G$ and
$h$ are continuous, $G(\mathbf{Y})h(\mathbf{Y})\to G(\mathsf{Y})h(\mathsf{Y})$
and $h(\mathbf{Y})\to h(\mathsf{Y})$ in probability as well. Both sequences are
bounded, by hypothesis on $G$ and by~\eqref{eq:tilt_bounds}, so the expectations
converge: $\E[G h]\to G(\mathsf{Y})h(\mathsf{Y})$ and $\E[h]\to h(\mathsf{Y})>0$.
Now~\eqref{eq:tilt_expectation} gives $\E_{\mathbf{a}}[G]\to G(\mathsf{Y})$
for every bounded continuous $G$, which implies convergence in probability
under the perturbed law.

\emph{(ii)}
Put $\Delta\mathbf{Y}\coloneqq\sqrt{\beta}\,(\mathbf{Y}-\mathsf{Y})$ and
\begin{equation*}
	\varepsilon(\mathbf{Y})\coloneqq\frac{h(\mathbf{Y})}{h(\mathsf{Y})}-1
	=\exp\left[\sum_{k=1}^{N-1}\sum_{i=1}^{k}
		\bigl(y_i^k-\mathsf{y}_i^k\bigr)(a_k-a_{k+1})\right]-1 .
\end{equation*}
By~\eqref{eq:tilt_bounds} we have $|\varepsilon|\le e^{2C_{\mathbf{a},\mathbf{z}}}$
on the polytope, and by part~(i) of \Cref{thm:GM_crystallization}
the unperturbed array satisfies $\mathbf{Y}\to\mathsf{Y}$ in probability.
Since $h$ is continuous, $\varepsilon(\mathbf{Y})\to0$ in probability, and being
bounded it converges to $0$ in $L^1$ as well. Therefore, for bounded continuous $G$,
\begin{equation*}
	\E_{\mathbf{a}}\bigl[G(\Delta\mathbf{Y})\bigr]
	=\frac{\E\bigl[G(\Delta\mathbf{Y})(1+\varepsilon)\bigr]}{\E\bigl[1+\varepsilon\bigr]}
	=\frac{\E\bigl[G(\Delta\mathbf{Y})\bigr]
		+O\bigl(\lVert G\rVert_\infty\ssp\E|\varepsilon|\bigr)}
		{1+O\bigl(\E|\varepsilon|\bigr)}
	\xrightarrow[\beta\to\infty]{} \E\bigl[G(\zeta)\bigr],
\end{equation*}
because $\Delta\mathbf{Y}\to\zeta$ in distribution under the unperturbed law by
part~(ii) of \Cref{thm:GM_crystallization}. This completes the proof.
\end{proof}

\begin{remark}[Statistical mechanics interpretation]
	The $\beta$-corners process can be viewed as particles at inverse temperature $\beta$
	subject to logarithmic repulsion and a confining potential.
	As $\beta\to\infty$ (zero temperature), particles freeze to their ground state positions,
	which are the roots of polynomial derivatives.
	The perturbation~$\mathbf{a}$ acts as an external field of strength $O(1)$,
	while the repulsion grows as $O(\beta)$; the uniform
	bound~\eqref{eq:tilt_bounds} is the quantitative form of this mismatch.
	Thus, the field is too weak to move the ground state.
	The dGFF (\Cref{def:dGFF}) captures thermal fluctuations around the ground state, at the scale
	$O(\beta^{-1/2})$ set by the energy barrier of order $\beta$.
	Again, the perturbation leaves the fluctuations unaffected.
\end{remark}

\section{Zero-temperature asymptotics: linearly growing perturbation}
\label{sec:beta_inf_scaled}

We now consider the regime where the perturbation grows linearly with $\beta$:
\begin{equation}
	\label{eq:scaled_perturbation}
	a_i = \frac{\beta}{2}\mathsf{a}_i, \qquad i=1,\ldots,N,
\end{equation}
for fixed parameters $\mathsf{a} = (\mathsf{a}_1,\ldots,\mathsf{a}_N)\in\R^N$.
In this scaling, the external field competes with the eigenvalue repulsion.
Throughout the subsection, 
we refer to the special case
$\mathsf{a}_1=\cdots=\mathsf{a}_N$ as the \emph{constant
perturbation}.

\subsection{The deformed lattice: existence and uniqueness}

In the growing perturbation regime,
the limiting crystal lattice is \emph{deformed} compared to the unperturbed case:

\begin{definition}
\label{def:deformed_lattice}
Given a top row $\mathbf{z}=(z_1>\cdots>z_N)$ and parameters
$\mathsf{a}=(\mathsf{a}_1,\ldots,\mathsf{a}_N)$,
the \emph{deformed $\infty$-corners lattice}
$\{\bar{\mathsf{y}}_i^k\}_{1\le i\le k\le N}$ is defined by
$\bar{\mathsf{y}}_i^N = z_i$ for $i=1,\ldots,N$, and the remaining
positions $\{\bar{\mathsf{y}}_i^k\}_{1\le i\le k\le N-1}$ are the unique 
solution of the \emph{coupled optimality equations}:
\begin{equation}
	\label{eq:deformed_optimality}
	\sum_{\substack{j=1\\j\ne i}}^{k}\frac{1}{\bar{\mathsf{y}}_i^k - \bar{\mathsf{y}}_j^k}
	- \frac{1}{2}\sum_{q=1}^{k+1}\frac{1}{\bar{\mathsf{y}}_i^k - \bar{\mathsf{y}}_q^{k+1}}
	- \frac{1}{2}\sum_{p=1}^{k-1}\frac{1}{\bar{\mathsf{y}}_i^k - \bar{\mathsf{y}}_p^{k-1}}
	= \frac{\mathsf{a}_k - \mathsf{a}_{k+1}}{2},
	\qquad i=1,\ldots,k,
\end{equation}
for $k=1,\ldots,N-1$, subject to \emph{strict}
interlacing.  That is, $\bar{\mathsf{y}}^k \prec
\bar{\mathsf{y}}^{k+1}$, and
$\bar{\mathsf{y}}^k_i\ne\bar{\mathsf{y}}^{k+1}_j$ for all
$1\le k\le N-1$, $1\le i\le k$ and $1\le j\le k+1$.
Equivalently, $\bar{\mathsf{Y}}\coloneqq
\{\bar{\mathsf{y}}_i^k\}_{1\le i\le k\le N-1}$ lies in
the open Gelfand--Tsetlin polytope, where every denominator
of~\eqref{eq:deformed_optimality} is nonzero.  Existence
and uniqueness 
of the solution to~\eqref{eq:deformed_optimality}
are established in
\Cref{prop:variational_deformed} below.
\end{definition}

\begin{remark}
	The system~\eqref{eq:deformed_optimality} couples all
	levels simultaneously: the equation for position
	$(i,k)$ involves levels $k-1$, $k$, and $k+1$. 
	For the constant perturbation,
	the right-hand side vanishes, and \eqref{eq:deformed_optimality}
	holds for the unperturbed lattice $\mathsf{Y}$
	(\Cref{def:deterministic_lattice}),
	thanks to 
	\eqref{eq:optimality_equations}--\eqref{eq:polynomial_identity3}.
	By uniqueness, the
	deformed lattice $\bar{\mathsf{Y}}$
	then coincides with 
	$\mathsf{Y}$.
\end{remark}

The coupled optimality
equations~\eqref{eq:deformed_optimality} arise as
the critical
point conditions for the 
\emph{effective Hamiltonian}\footnote{For $N=1$ the array $\mathbf{Y}$ is
empty, so $H_{\mathsf{a}}\equiv0$, and the deformed lattice
is the top row itself. In the rest of this
subsection we thus take $N\ge2$.}
\begin{equation}
	\label{eq:effective_hamiltonian}
	H_{\mathsf{a}}(\mathbf{Y}) \coloneqq \sum_{k=1}^{N-1}\left[
		\sum_{1\le i<j\le k}\log(y_i^k - y_j^k)
		- \frac{1}{2}\sum_{p=1}^k\sum_{q=1}^{k+1}\log|y_p^k - y_q^{k+1}|
		- \frac{\mathsf{a}_k - \mathsf{a}_{k+1}}{2}\sum_{i=1}^k y_i^k
	\right],
\end{equation}
in which the top row $\mathbf{y}^N=\mathbf{z}$ is fixed
and $\mathbf{Y}=(\mathbf{y}^1,\ldots,\mathbf{y}^{N-1})$
ranges over the Gelfand--Tsetlin polytope
$\{\mathbf{y}^1\prec\cdots\prec\mathbf{y}^{N-1}\prec\mathbf{z}\}$.
It is read off from the joint density: at the fixed top
row, and under the
scaling~\eqref{eq:scaled_perturbation}, the
$\mathbf{Y}$-dependent part
of~\eqref{eq:perturbed_corners_density} (see also
\eqref{eq:perturbed_corners_density_explicit}) is exactly
$e^{-\beta H_{\mathsf{a}}(\mathbf{Y})}$ times a factor
not involving $\beta$.  

In the rest of this subsection, we establish two
properties of $H_{\mathsf{a}}$: it tends to $+\infty$ at
every boundary point of the polytope, and it is strictly
convex in the interior.  From these properties we deduce
the existence and uniqueness of the deformed lattice, and
later in \Cref{sub:crystallization_scaled}
we use them for the zero-temperature asymptotic analysis.

\begin{lemma}
\label{lem:boundary_blowup}
Let
\begin{equation}
	\label{eq:min_cross_gap}
	\rho(\mathbf{Y})\coloneqq\min_{1\le k\le N-1}\ \min_{1\le p\le k}\
	\min_{1\le q\le k+1}\bigl|y_p^k-y_q^{k+1}\bigr|,
	\qquad \mathbf{y}^N=\mathbf{z},
\end{equation}
denote the smallest cross-level gap of the array. There is a constant
$C=C(N,\mathbf{z},\mathsf{a})$ such that
\begin{equation}
	\label{eq:boundary_bound}
	H_{\mathsf{a}}(\mathbf{Y})\ \ge\ \frac12\ssp\log\frac{1}{\rho(\mathbf{Y})}-C
\end{equation}
at every point $\mathbf{Y}$ of the open Gelfand--Tsetlin polytope. In particular
$H_{\mathsf{a}}(\mathbf{Y})\to+\infty$ as $\mathbf{Y}$ approaches any
boundary point of the polytope, and every sublevel set of $H_{\mathsf{a}}$ is
a compact subset of its interior.
\end{lemma}

\begin{proof}
The linear term of~\eqref{eq:effective_hamiltonian} is bounded on the bounded
polytope, so it suffices to bound its logarithmic part, which is
$-\log\Psi(\mathbf{Y})$ for
\begin{equation}
	\label{eq:geometric_factor}
	\Psi(\mathbf{Y})\coloneqq\prod_{k=1}^{N-1}
	\frac{\prod_{p=1}^{k}\prod_{q=1}^{k+1}\bigl|y_p^k-y_q^{k+1}\bigr|^{1/2}}
	{\prod_{1\le i<j\le k}(y_i^k-y_j^k)}\ssp.
\end{equation}
For a strictly decreasing $\mathbf{w}\in\R^{n}$ write
$\gamma(\mathbf{w})=\min_{p}(w_p-w_{p+1})$ for its smallest gap.
We also display the top row, writing $\Psi(\mathbf{Y};\mathbf{w})$ and
$\rho(\mathbf{Y};\mathbf{w})$ for the
quantities~\eqref{eq:geometric_factor} and~\eqref{eq:min_cross_gap} formed
from an array $\mathbf{y}^1\prec\cdots\prec\mathbf{y}^{n-1}\prec\mathbf{w}$,
so that $\Psi(\mathbf{Y})=\Psi(\mathbf{Y};\mathbf{z})$ and
$\rho(\mathbf{Y})=\rho(\mathbf{Y};\mathbf{z})$.
We prove the scale invariant bound
\begin{equation}
	\label{eq:geometric_bound}
	\Psi(\mathbf{Y};\mathbf{w})\ \le\
	V(\mathbf{w})\ssp\sqrt{\rho(\mathbf{Y};\mathbf{w})/\gamma(\mathbf{w})},
	\qquad n\ge2,
\end{equation}
by induction on the number of levels; taking $n=N$, $\mathbf{w}=\mathbf{z}$
and logarithms gives~\eqref{eq:boundary_bound}.

\medskip
Let $\mathbf{u}\in\R^{n-1}$ interlace
$\mathbf{w}\in\R^{n}$ and put
$\Pi=\prod_{p=1}^{n-1}\prod_{q=1}^{n}|u_p-w_q|$. For $1\le p\le n-1$ and
$1\le q\le n$ set $B_{p,q}=w_q-w_{p+1}$ if $q\le p$ and $B_{p,q}=w_p-w_q$ if
$q\ge p+1$. Each $B_{p,q}$ is a difference $w_i-w_j$ with $i<j$, so
$B_{p,q}\ge\gamma(\mathbf{w})$, and $w_{p+1}<u_p<w_p$ gives
$|u_p-w_q|<B_{p,q}$ in both cases. As $(p,q)$ runs over its range, the pairs
$(q,p+1)$ with $q\le p$ and the pairs $(p,q)$ with $q\ge p+1$ each run exactly
once through all pairs $i<j$ in $\{1,\ldots,n\}$, so that
$\prod_{p,q}B_{p,q}=V(\mathbf{w})^2$. Bounding all factors of $\Pi$ but one by
their $B_{p,q}$, and the omitted one from below by $\gamma(\mathbf{w})$,
therefore gives, for every $(p_0,q_0)$,
\begin{equation*}
	\Pi\ \le\ |u_{p_0}-w_{q_0}|\ssp\frac{V(\mathbf{w})^{2}}{\gamma(\mathbf{w})}\ssp.
\end{equation*}
Minimizing over $(p_0,q_0)$ yields
\begin{equation}
	\label{eq:two_row_gap}
	\Pi\ \le\ \frac{V(\mathbf{w})^{2}}{\gamma(\mathbf{w})}\ssp
	\min_{p,q}|u_p-w_q|\ssp,
\end{equation}
while choosing $p_0$ with $u_{p_0}-u_{p_0+1}=\gamma(\mathbf{u})$ and
$q_0=p_0+1$, for which $u_{p_0}-w_{p_0+1}<u_{p_0}-u_{p_0+1}$ because
$u_{p_0+1}<w_{p_0+1}$, yields
\begin{equation}
	\label{eq:two_row_mingap}
	\Pi\ \le\ \frac{V(\mathbf{w})^{2}}{\gamma(\mathbf{w})}\ssp\gamma(\mathbf{u})
	\qquad (n\ge3).
\end{equation}

\medskip
For $n=2$ the array is a single point, and
\eqref{eq:geometric_bound} says that the larger of its two gaps to
$\mathbf{w}$ is at most their sum.
Let $n\ge3$, assume~\eqref{eq:geometric_bound} for $n-1$, and peel off the
top bond, writing $\mathbf{u}=\mathbf{y}^{n-1}$ and
$\mathbf{Y}'=(\mathbf{y}^1,\ldots,\mathbf{y}^{n-2})$:
\begin{equation*}
	\Psi(\mathbf{Y};\mathbf{w})
	=\Psi(\mathbf{Y}';\mathbf{u})\ssp\frac{\Pi^{1/2}}{V(\mathbf{u})}
	\ \le\ \sqrt{\frac{\rho(\mathbf{Y}';\mathbf{u})\ssp\Pi}
	{\gamma(\mathbf{u})}}\ssp,
	\qquad
	\rho(\mathbf{Y};\mathbf{w})
	=\min\bigl\{\rho(\mathbf{Y}';\mathbf{u}),\ \min_{p,q}|u_p-w_q|\bigr\}.
\end{equation*}
If the minimum is the first term, \eqref{eq:two_row_mingap} bounds the display
by $V(\mathbf{w})\sqrt{\rho(\mathbf{Y};\mathbf{w})/\gamma(\mathbf{w})}$;
otherwise \eqref{eq:two_row_gap} bounds it by the same quantity times
$\sqrt{\rho(\mathbf{Y}';\mathbf{u})/\gamma(\mathbf{u})}$.
The latter factor is at most $1$
because the point of $\mathbf{y}^{n-2}$ inside the shortest gap of
$\mathbf{u}$ lies within $\gamma(\mathbf{u})$ of both its ends.
This proves~\eqref{eq:geometric_bound}.

Finally, $\rho(\mathbf{Y})\to0$ exactly when $\mathbf{Y}$ approaches the
boundary of the polytope, and $\{H_{\mathsf{a}}\le M\}$ is contained in
$\{\rho\ge e^{-2(M+C)}\}$, a compact subset of the interior.
\end{proof}

To address the convexity of $H_{\mathsf{a}}$, 
we need the following lemma:
\begin{lemma}
\label{lem:alternating}
Let $n\ge2$, let $t_1>\cdots>t_n$ or $t_1<\cdots<t_n$, and let
$u\in\R^n$. Then
\begin{equation}
	\label{eq:alternating_form}
	\sum_{1\le i<j\le n}(-1)^{i+j}\,\frac{(u_i-u_j)^2}{(t_i-t_j)^2}
	\;\le\;0,
\end{equation}
with equality if and only if $u_1=\cdots=u_n$.
\end{lemma}
Observe that the sum in~\eqref{eq:alternating_form} 
is not invariant under permutations of the $t_i$'s
due to the signs $(-1)^{i+j}$, so we assume the $t_i$'s to be monotone. In fact, the statement becomes false if the $t_i$'s are not ordered.

\begin{proof}[Proof of \Cref{lem:alternating}]
Assume $t_1<\cdots<t_n$, the other case is equivalent.
Let
\begin{equation*}
	\Phi(t)\coloneqq\sum_{1\le i<j\le n}(-1)^{i+j}\log(t_j-t_i),
\end{equation*}
and move the points along $t_i(s)=t_i+s\ssp u_i$, the ordering persisting for
small $|s|$. Since
\begin{equation*}
\frac{d^{2}}{ds^{2}}\log\bigl(t_j(s)-t_i(s)\bigr)\big|_{s=0}
=-(u_j-u_i)^{2}/(t_j-t_i)^{2}, 
\end{equation*}
we have
\begin{equation}
	\label{eq:alt_convexity}
	\frac{d^{2}}{ds^{2}}\ssp\Phi\bigl(t(s)\bigr)\Big|_{s=0}
	=-\sum_{1\le i<j\le n}(-1)^{i+j}\ssp
	\frac{(u_i-u_j)^{2}}{(t_i-t_j)^{2}}.
\end{equation}

\medskip
\noindent\textbf{Step 1.}
Put $m\coloneqq\lfloor n/2\rfloor$ and
$\epsilon\coloneqq n-2m\in\{0,1\}$, and separate the odd and even positions,
\begin{equation*}
	a_i\coloneqq t_{2i-1},\quad \alpha_i\coloneqq u_{2i-1}
	\quad(1\le i\le m+\epsilon),
	\qquad
	b_j\coloneqq t_{2j},\quad \beta_j\coloneqq u_{2j}
	\quad(1\le j\le m),
\end{equation*}
so that $a_1<b_1<a_2<\cdots$. Two indices have equal parity exactly when they
lie in the same family, so
(recall that $V(\cdot)$ denotes the Vandermonde)
\begin{equation*}
	\Phi=\log V(a)+\log V(b)-\log\prod_{i,j}|a_i-b_j|\ssp.
\end{equation*}

\medskip
\noindent\textbf{Step 2.}
Let $P(z)\coloneqq\prod_i(z-a_i)$,
$Q(z)\coloneqq\prod_j(z-b_j)$, $R\coloneqq P/Q$, and set
\begin{equation*}
	c_j\coloneqq\frac{|P(b_j)|}{|Q'(b_j)|}>0,
	\qquad
	d_i\coloneqq\frac{|Q(a_i)|}{|P'(a_i)|}>0\ssp.
\end{equation*}
Counting negative factors in each product gives
$P(b_j)/Q'(b_j)=(-1)^{\epsilon}c_j$ and
$Q(a_i)/P'(a_i)=(-1)^{\epsilon+1}d_i$. Since $\deg P-\deg Q=\epsilon$,
the partial fraction expansion of $R$ reads
\begin{equation}
	\label{eq:alt_R_expansion}
	R(z)=\epsilon z+\mathrm{const}+(-1)^{\epsilon}\sum_{j=1}^{m}\frac{c_j}{z-b_j},
\end{equation}
and by definition we have
\begin{equation}
	\label{eq:alt_d_inverse}
	\frac{1}{d_i}=(-1)^{\epsilon+1}R'(a_i)\ssp.
\end{equation}

Let $i\ne k$. Both $a_i$ and $a_k$ are roots of $P$,
hence of $R$, so the divided difference
$\frac{R(a_i)-R(a_k)}{a_i-a_k}$ vanishes. 
We take the difference $R(a_i)-R(a_k)$ using \eqref{eq:alt_R_expansion} and
$\frac{1}{a_i-b_j}-\frac{1}{a_k-b_j}
=\frac{a_k-a_i}{(a_i-b_j)(a_k-b_j)}$, and get
\begin{equation*}
	0=\epsilon-(-1)^{\epsilon}
	\sum_{j=1}^{m}\frac{c_j}{(a_i-b_j)(a_k-b_j)}\ssp.
\end{equation*}
For $i=k$, differentiating the same expansion at $a_i$ gives
\begin{equation*}
	R'(a_i)=\epsilon-(-1)^{\epsilon}
	\sum_{j=1}^{m}\frac{c_j}{(a_i-b_j)^{2}}\ssp.
\end{equation*}
Multiplying both by $(-1)^{\epsilon+1}$, and using
$(-1)^{\epsilon+1}\epsilon=\epsilon$
(recall $\epsilon\in\{0,1\}$), together with
\eqref{eq:alt_d_inverse}, the two cases combine into
\begin{equation}
	\label{eq:alt_orthogonality}
	\epsilon+\sum_{j=1}^{m}\frac{c_j}{(a_i-b_j)(a_k-b_j)}
	=\frac{\mathbf{1}_{i=k}}{d_i}\ssp.
\end{equation}
Multiplying~\eqref{eq:alt_orthogonality} by $\sqrt{d_id_k}$ and setting
\begin{equation}
	\label{eq:alt_O_def}
	O_{i0}\coloneqq\sqrt{\epsilon\ssp d_i},
	\qquad
	O_{ij}\coloneqq\frac{\sqrt{d_ic_j}}{a_i-b_j}
	\quad(1\le j\le m),
\end{equation}
turns it into $OO^{\mathsf{T}}=I$. For $\epsilon=1$ the matrix $O$ is square
of size $m+1$; for $\epsilon=0$ its zeroth column vanishes and the remaining
$m\times m$ block is orthogonal. In either case the matrix formed by the
nonvanishing columns is square and orthogonal, so with
$w_{ij}\coloneqq O_{ij}^{2}=d_ic_j/(a_i-b_j)^{2}>0\ (1\le j\le m)$ both the rows and the
columns give
\begin{equation}
	\label{eq:alt_stochastic}
	\epsilon\ssp d_i+\sum_{j=1}^{m}w_{ij}=1,
	\qquad
	\sum_{i=1}^{m+\epsilon}w_{ij}=1\ssp.
\end{equation}
Finally,
\begin{equation*}
	\prod_j|P(b_j)|=\prod_i|Q(a_i)|=\prod_{i,j}|a_i-b_j|,
	\qquad
	\prod_j|Q'(b_j)|=V(b)^{2},
	\qquad
	\prod_i|P'(a_i)|=V(a)^{2},
\end{equation*}
so that
\begin{equation}
	\label{eq:alt_Phi_cd}
	\prod_j c_j=\frac{\prod_{i,j}|a_i-b_j|}{V(b)^{2}},
	\qquad
	\prod_i d_i=\frac{\prod_{i,j}|a_i-b_j|}{V(a)^{2}},
	\qquad 
	\Phi=-\frac12\sum_i\log d_i-\frac12\sum_j\log c_j\ssp.
\end{equation}

\medskip
\noindent\textbf{Step 3.}
Let the points move as above, so that
$a_i$, $b_j$, $c_j$, $d_i$, $w_{ij}$ become functions of $s$; dots denote
$d/ds$ at $s=0$. Since \eqref{eq:alt_orthogonality}
and~\eqref{eq:alt_stochastic} were derived at an arbitrary point of the chamber,
and the chamber is preserved for small $|s|$, they hold as identities in $s$ and
may be differentiated. Put
\begin{equation*}
	X_{ij}\coloneqq\frac{\alpha_i-\beta_j}{a_i-b_j},
	\quad
	\sigma_i\coloneqq\frac{d}{ds}\log d_i,
	\quad
	\kappa_j\coloneqq\frac{d}{ds}\log c_j,
	\quad
	Y_{ij}\coloneqq\frac{\sigma_i+\kappa_j}{2}-X_{ij}.
\end{equation*}
From $\log|O_{ij}|=\frac12\log d_i+\frac12\log c_j-\log|a_i-b_j|$
(where $O$ is defined in~\eqref{eq:alt_O_def}) and
$\frac{d}{ds}X_{ij}=-X_{ij}^{2}$ we get, for $1\le j\le m$,
\begin{equation*}
	\frac{d}{ds}\log|O_{ij}|=Y_{ij},
	\qquad
	\frac{d^{2}}{ds^{2}}\log|O_{ij}|
	=\tfrac12\dot\sigma_i+\tfrac12\dot\kappa_j+X_{ij}^{2},
\end{equation*}
while for $\epsilon=1$, where $O_{i0}=\sqrt{d_i}$, one has
$\frac{d}{ds}\log O_{i0}=\frac12\sigma_i$ and
$\frac{d^{2}}{ds^{2}}\log O_{i0}=\frac12\dot\sigma_i$; for $\epsilon=0$ the
zeroth column vanishes identically and contributes nothing. Since
$\frac{d^{2}}{ds^{2}}e^{2g}=2e^{2g}\bigl(g''+2(g')^{2}\bigr)$, differentiating
the first relation of~\eqref{eq:alt_stochastic} twice, dividing by $2$ and
summing over $i$, and then using both relations
of~\eqref{eq:alt_stochastic} to collapse the resulting weighted sums of
$\dot\sigma_i$ and $\dot\kappa_j$, gives
\begin{equation*}
	0=\frac12\sum_i\dot\sigma_i+\frac12\sum_j\dot\kappa_j
	+\frac{\epsilon}{2}\sum_id_i\ssp\sigma_i^{2}
	+\sum_{i,j}w_{ij}X_{ij}^{2}
	+2\sum_{i,j}w_{ij}Y_{ij}^{2}.
\end{equation*}
Differentiating the last identity 
in~\eqref{eq:alt_Phi_cd} twice yields
$\frac{d^{2}}{ds^{2}}\Phi=-\frac12\sum_i\dot\sigma_i-\frac12\sum_j\dot\kappa_j$,
so that
\begin{equation}
	\label{eq:alt_certificate}
	\frac{d^{2}}{ds^{2}}\ssp\Phi\bigl(t(s)\bigr)\Big|_{s=0}
	=\sum_{i,j}w_{ij}X_{ij}^{2}
	+2\sum_{i,j}w_{ij}Y_{ij}^{2}
	+\frac{\epsilon}{2}\sum_id_i\ssp\sigma_i^{2}.
\end{equation}
All the $w_{ij}$ and $d_i$ are positive, so the right-hand side
of~\eqref{eq:alt_certificate} is nonnegative; by~\eqref{eq:alt_convexity} this
is the asserted inequality.

\medskip
\noindent\textbf{Step 4.}
Finally, if the left-hand side of~\eqref{eq:alt_certificate}
vanishes, then $\sum_{i,j}w_{ij}X_{ij}^{2}=0$, and since every $w_{ij}$ is
positive, $X_{ij}=0$, that is $\alpha_i=\beta_j$, for all $i$ and $j$. Both
families are nonempty because $n\ge2$, so varying $i$ at fixed $j$ and then $j$
at fixed $i$ shows that all the $\alpha_i$ and all the $\beta_j$ equal one
common value, that is, $u_1=\cdots=u_n$. Conversely, a constant $u$ makes every
difference $u_i-u_j$ vanish.
\end{proof}

Let us denote the Hessian 
of the 
 effective Hamiltonian $H_{\mathsf{a}}$
\eqref{eq:effective_hamiltonian}
by 
\begin{equation}
	\label{eq:Q_form}
	Q_{\mathbf{Y}}(\mathbf{v})=
	-\frac{1}{2}
\sum_{i,k}\sum_{j,\ell}v_i^k\,
	\frac{\partial^2 H_{\mathsf{a}}}{\partial y_i^k\,\partial y_j^\ell}\,
	v_j^\ell
	=
	\sum_{k=1}^{N-1}\left[
		\sum_{1\le i<j\le k}\frac{(v_i^k-v_j^k)^2}{2(y_i^k-y_j^k)^2}
		-\sum_{p=1}^{k}\sum_{q=1}^{k+1}
			\frac{(v_p^k-v_q^{k+1})^2}{4(y_p^k-y_q^{k+1})^2}
	\right],
\end{equation}
acting on fields $\mathbf{v}=\{v_i^k\}$ with $v^N\equiv0$.
Here $\mathbf{Y}$ is an interior point of the Gelfand--Tsetlin polytope, so that all denominators are nonzero.
The second formula 
in \eqref{eq:Q_form}
readily follows
from \eqref{eq:effective_hamiltonian}.
In particular, the Hessian of $H_{\mathsf{a}}$ does not depend on
$\mathsf{a}$.

\begin{proposition}
\label{prop:strict_convexity}
The quadratic form $Q_{\mathbf{Y}}$ is negative definite, so
$H_{\mathsf{a}}$ is strictly convex on the Gelfand--Tsetlin polytope.
\end{proposition}
\begin{proof}
For a bond $(k,k+1)$, merge the two levels into the single decreasing
sequence
$t_1=y_1^{k+1}>t_2=y_1^{k}>t_3=y_2^{k+1}>\cdots>t_{2k+1}=y_{k+1}^{k+1}$,
and let $u\in\R^{2k+1}$ carry the values of $\mathbf{v}$ at the corresponding
positions, that is, $u_{2r-1}=v_r^{k+1}$ for $r=1,\ldots,k+1$ and
$u_{2r}=v_r^{k}$ for $r=1,\ldots,k$. Pairs of equal
parity are exactly the same-level pairs, and pairs of opposite parity the
cross-level ones, so \Cref{lem:alternating} with $n=2k+1$ states that
\begin{equation*}
	B_k \coloneqq
	\sum_{1\le i<j\le k}\frac{(v_i^k-v_j^k)^2}{(y_i^k-y_j^k)^2}
	+\sum_{1\le i<j\le k+1}
	\frac{(v_i^{k+1}-v_j^{k+1})^2}{(y_i^{k+1}-y_j^{k+1})^2}
	-\sum_{p=1}^{k}\sum_{q=1}^{k+1}
	\frac{(v_p^k-v_q^{k+1})^2}{(y_p^k-y_q^{k+1})^2}
	\;\le\;0,
\end{equation*}
with equality only if the entries of $v^k$ and $v^{k+1}$ are all equal to one
another. Level $1$ carries no same-level pair and $v^N\equiv0$, so each
same-level sum with $2\le k\le N-1$ appears in exactly two of the $B_k$ and
the boundary sums vanish; summing over the bonds therefore telescopes to
\begin{equation*}
	\sum_{k=1}^{N-1}B_k
	=2\sum_{k=1}^{N-1}\sum_{1\le i<j\le k}
	\frac{(v_i^k-v_j^k)^2}{(y_i^k-y_j^k)^2}
	-\sum_{k=1}^{N-1}\sum_{p=1}^{k}\sum_{q=1}^{k+1}
	\frac{(v_p^k-v_q^{k+1})^2}{(y_p^k-y_q^{k+1})^2}
	=4\ssp Q_{\mathbf{Y}}(\mathbf{v})\le0.
\end{equation*}
If $Q_{\mathbf{Y}}(\mathbf{v})=0$, then every $B_k$
vanishes, so for each $k$ the entries of $v^k$ and
$v^{k+1}$ are all equal to one another; consecutive bonds
overlap in a level, so all entries equal a single
constant, which is~$0$ because $v^N\equiv0$. Hence
$Q_{\mathbf{Y}}$ is negative definite and,
by~\eqref{eq:Q_form}, $H_{\mathsf{a}}$ is strictly
convex.
\end{proof}

Putting together \Cref{lem:boundary_blowup}
and \Cref{prop:strict_convexity}, we obtain:
\begin{proposition}
\label{prop:variational_deformed}
The deformed lattice
$\bar{\mathsf{Y}}=\{\bar{\mathsf{y}}_i^k\}$
of \Cref{def:deformed_lattice} exists, is unique, and is
the unique critical point, as well as the global minimizer, of
the effective Hamiltonian
$H_{\mathsf{a}}$ \eqref{eq:effective_hamiltonian}
in the Gelfand--Tsetlin polytope.
\end{proposition}

\begin{remark}
	\label{rmk:convexity_vs_GM}
	Let us compare our variational argument with the
	level-by-level one of~\cite{GorinMarcus2020}*{Section~3.4},
	which in the unperturbed case established that the dGFF of
	\Cref{def:dGFF} is well defined, that is, that its
	density~\eqref{eq:dGFF_density} is integrable.

	At
	$\mathbf{Y}=\mathsf{Y}$, the
	$\infty$-corners lattice of
	\Cref{def:deterministic_lattice}, the negative
	definiteness of $Q_{\mathbf{Y}}$ is already contained
	in \Cref{rmk:dGFF_proper}. That argument integrates the
	levels one at a time, and each step uses that the level
	being integrated consists of the critical points of the
	polynomial with roots at the level above. Neither the
	deformed lattice nor a general point $\mathbf{Y}$
	of the interlacing
	polytope admits such a description, and strict
	convexity requires the Hessian condition at every point
	rather than at a lattice inside the polytope.
	\Cref{prop:strict_convexity} provides such a convexity
	statement.
\end{remark}

\subsection{Law of Large Numbers and Central Limit Theorem}
\label{sub:crystallization_scaled}

Under the scaling~\eqref{eq:scaled_perturbation}, the $\mathbf{Y}$-dependent
part of the perturbation in the joint
density~\eqref{eq:perturbed_corners_density} becomes:
\begin{equation*}
\exp\left[a_N\sum_{i=1}^N z_i + \sum_{k=1}^{N-1}\sum_{i=1}^k y_i^k(a_k - a_{k+1})\right]
	= \exp\left[\frac{\beta}{2}\left(\mathsf{a}_N\sum_{i=1}^N z_i
	+ \sum_{k=1}^{N-1}\sum_{i=1}^k y_i^k(\mathsf{a}_k - \mathsf{a}_{k+1})\right)\right].
\end{equation*}
This is now the same order as the repulsion terms.
The
remaining perturbation-dependent factors
of~\eqref{eq:perturbed_corners_density}, namely
$e^{-\frac12\sum_ia_i^2}$ and
$B_{\mathbf{a}}(\mathbf{z};\beta)$, do not depend on
$\mathbf{Y}$ and drop out of the conditional density
below.

By the definition of the perturbed corners process
(\Cref{def:perturbed_beta_corners}),
the factor $B_{\mathbf{a}}(\mathbf{z};\beta)$
in the perturbed G$\beta$E density~\eqref{eq:density_perturbed_GbetaE}
cancels with $1/B_{\mathbf{a}}(\mathbf{z};\beta)$
in the perturbed links~\eqref{eq:perturbed_corners_conditional}.
Since the top row is fixed, we pass to the conditional density of
$\mathbf{Y}$ given $\mathbf{z}$, discarding the factors that depend on
$\mathbf{z}$ alone.
Extracting the $\beta$-dependent terms from the
unperturbed links~\eqref{eq:unperturbed_links} and the perturbation,
the conditional density takes the exact form
\begin{equation}
	\label{eq:effective_density}
	p_\beta(\mathbf{Y}\mid\mathbf{z})
	= \frac{1}{Z_\beta}\,R(\mathbf{Y})\,
	\exp\left[-\beta \cdot H_{\mathsf{a}}(\mathbf{Y})\right],
	\qquad
	R(\mathbf{Y}) \coloneqq \prod_{k=1}^{N-1}
	\frac{\prod_{1\le i<j\le k}(y_i^k-y_j^k)^{2}}
	{\prod_{p=1}^{k}\prod_{q=1}^{k+1}\bigl|y_p^k-y_q^{k+1}\bigr|},
\end{equation}
for $\mathbf{Y}$ in the Gelfand--Tsetlin polytope,
where $H_{\mathsf{a}}$ is given by~\eqref{eq:effective_hamiltonian},
the residual factor $R$ does not depend on $\beta$,
and the normalizing constant $Z_\beta$ depends on $\beta$, $\mathbf{z}$,
and $\mathsf{a}$.
Clearly, on compact subsets of the open Gelfand--Tsetlin polytope, $R$ is bounded
between two positive constants.

Differentiating $H_{\mathsf{a}}$ with respect to $y_i^k$
and setting the result to zero recovers the coupled
optimality equations~\eqref{eq:deformed_optimality} from
\Cref{def:deformed_lattice}.

\begin{theorem}[Crystallization for scaled perturbation]
\label{thm:crystallization_scaled}
Let $\{y_i^k(\beta)\}_{1\le i\le k\le N}$ be the perturbed $\beta$-corners process
with scaled perturbation $a_i = \frac{\beta}{2}\mathsf{a}_i$ and fixed top row
$\mathbf{z}=(z_1>\cdots>z_N)$. Let $\{\bar{\mathsf{y}}_i^k\}$ be the deformed
$\infty$-corners lattice from \Cref{def:deformed_lattice}.
Then as $\beta\to\infty$:
\begin{equation*}
y_i^k(\beta) \to \bar{\mathsf{y}}_i^k \qquad \text{in probability}.
\end{equation*}
\end{theorem}

Both the theorem and the central limit theorem below follow from the
following tail bound at the fluctuation scale $\beta^{-1/2}$.

\begin{lemma}
\label{lem:gaussian_tail_scaled}
In the setting of \Cref{thm:crystallization_scaled} there are
$C,c,\beta_1>0$, depending only on $N$, $\mathbf{z}$ and $\mathsf{a}$,
such that
\begin{equation}
	\label{eq:gaussian_tail_scaled}
	\P\Bigl(\sqrt{\beta}\ssp\max_{i,k}|y_i^k(\beta)-\bar{\mathsf{y}}_i^k|\ge s\Bigr)
	\le C\ssp e^{-c\ssp s^{2}}
	\qquad\text{for all }s>0\text{ and }\beta\ge\beta_1.
\end{equation}
\end{lemma}

\begin{proof}
Write $H_0=H_{\mathsf{a}}(\bar{\mathsf{Y}})$, and for $\varepsilon>0$ put
$K_\varepsilon=\{\max_{i,k}|y_i^k-\bar{\mathsf{y}}_i^k|\ge\varepsilon\}$ and
$\delta(\varepsilon)=\inf_{K_\varepsilon}H_{\mathsf{a}}-H_0$, which is
positive because the infimum is attained (\Cref{lem:boundary_blowup}) and the
minimizer is unique (\Cref{prop:variational_deformed}).
Fix $\beta_1>0$. On $K_\varepsilon$ we have
$e^{-\beta H_{\mathsf{a}}}\le
e^{-(\beta-\beta_1)(H_0+\delta(\varepsilon))}\ssp e^{-\beta_1H_{\mathsf{a}}}$,
and $\int R\ssp e^{-\beta_1H_{\mathsf{a}}}<\infty$ because the integrand is a
constant multiple of the perturbed
link~\eqref{eq:perturbed_corners_conditional} at $\beta_1$; near
$\bar{\mathsf{Y}}$, where $R$ is bounded below and
$H_{\mathsf{a}}\le H_0+\delta(\varepsilon)/2$, we get
$Z_\beta\ge c\ssp e^{-\beta(H_0+\delta(\varepsilon)/2)}$. Together these give
\begin{equation}
	\label{eq:concentration_scaled}
	\P\bigl(K_\varepsilon\bigr)
	\le C(\varepsilon)\ssp e^{-\beta\ssp\delta(\varepsilon)/2},
	\qquad \beta\ge2\beta_1.
\end{equation}
By \Cref{prop:strict_convexity} the Hessian of $H_{\mathsf{a}}$ at
$\bar{\mathsf{Y}}$ is positive definite, so on a ball
$B_{r_0}(\bar{\mathsf{Y}})$ inside the Gelfand--Tsetlin polytope the
difference $H_{\mathsf{a}}-H_0$ is comparable to
$|\mathbf{Y}-\bar{\mathsf{Y}}|^{2}$, and there $R$ is bounded above and below.
Restricting the integral defining $Z_{\beta}$ to $B_{r_0}(\bar{\mathsf{Y}})$
and rescaling by $\sqrt{\beta}$ gives
$Z_{\beta}\ge c\ssp\beta^{-N(N-1)/4}e^{-\beta H_{0}}$.
Comparing the numerator and the normalization
of~\eqref{eq:effective_density} with Gaussian integrals
gives~\eqref{eq:gaussian_tail_scaled} for $s\le r_0\sqrt{\beta}$. For larger
$s$ we use~\eqref{eq:concentration_scaled} with $\varepsilon=r_0$; since the
Gelfand--Tsetlin polytope is bounded, the event
in~\eqref{eq:gaussian_tail_scaled} is empty unless $s^{2}=O(\beta)$, so that
bound is again of the form $C\ssp e^{-c\ssp s^{2}}$.
\end{proof}

\begin{proof}[Proof of \Cref{thm:crystallization_scaled}]
Apply \Cref{lem:gaussian_tail_scaled} with $s=\varepsilon\sqrt{\beta}$: for
every $\varepsilon>0$ one gets
\begin{equation*}
\P\bigl(\max_{i,k}|y_i^k(\beta)-\bar{\mathsf{y}}_i^k|\ge\varepsilon\bigr)
\le C\ssp e^{-c\ssp\varepsilon^{2}\beta}\to0,
\end{equation*}
which is the desired convergence in probability
(moreover, with an exponential rate).
\end{proof}

The fluctuations 
of the perturbed $\beta$-corners process
$\mathbf{Y}$
around the limiting lattice
$\bar{\mathsf{Y}}$
are again Gaussian,
but with a perturbed covariance structure.
We need the following object.

\begin{definition}[Deformed discrete Gaussian Free Field]
\label{def:deformed_dGFF}
The \emph{deformed dGFF} on the lattice $\{\bar{\mathsf{y}}_i^k\}$
is the centered Gaussian vector $\{\bar{\zeta}_i^k\}_{1\le i\le k\le N}$
with $\bar{\zeta}_i^N=0$ (pinned top row) and joint density proportional to
\begin{equation}
\label{eq:deformed_dGFF_density}
\exp\left(-\frac{1}{2}\sum_{i,k}\sum_{j,\ell} \bar{\zeta}_i^k \cdot \mathcal{H}_{(i,k),(j,\ell)} \cdot \bar{\zeta}_j^\ell\right),
\end{equation}
where $\mathcal{H}$ is the Hessian matrix of the effective Hamiltonian~\eqref{eq:effective_hamiltonian}
evaluated at the deformed lattice:
\begin{equation*}
\mathcal{H}_{(i,k),(j,\ell)} = \frac{\partial^2 H_{\mathsf{a}}}{\partial y_i^k \partial y_j^\ell}\bigg|_{\mathbf{Y}=\bar{\mathsf{Y}}}.
\end{equation*}
The matrix $\mathcal{H}=-2\ssp Q_{\bar{\mathsf{Y}}}$ 
given by \eqref{eq:Q_form}
is positive definite by
\Cref{prop:strict_convexity}, so the density
\eqref{eq:deformed_dGFF_density}
defines a centered
Gaussian vector.
\end{definition}

\begin{remark}
The quadratic form in the exponent
of~\eqref{eq:deformed_dGFF_density} is the same function
of the underlying lattice as in the unperturbed
case~\eqref{eq:dGFF_density}.
What changes is the lattice at which it is
evaluated, namely,~$\bar{\mathsf{Y}}$ in place of
$\mathsf{Y}$, and so $Q_{\bar{\mathsf{Y}}}$ depends on
the perturbation parameters 
$\mathsf{a}$ through the deformed
lattice~$\bar{\mathsf{Y}}$.
\end{remark}

\begin{theorem}[CLT for scaled perturbation]
\label{thm:CLT_scaled}
Under the assumptions of \Cref{thm:crystallization_scaled},
we have
\begin{equation*}
\sqrt{\beta}\,(y_i^k(\beta) - \bar{\mathsf{y}}_i^k) \xrightarrow{d} \bar{\zeta}_i^k,
\qquad \beta\to\infty,
\end{equation*}
jointly
for all $1\le i\le k\le N$,
where $\{\bar{\zeta}_i^k\}$ is the deformed dGFF from \Cref{def:deformed_dGFF}.
\end{theorem}

\begin{proof}
Substituting
$y_i^k = \bar{\mathsf{y}}_i^k + \frac{\Delta y_i^k}{\sqrt{\beta}}$
into~\eqref{eq:effective_hamiltonian} and expanding
as $\beta\to\infty$,
we get
\begin{equation*}
	H_{\mathsf{a}}(\mathbf{Y}) = H_{\mathsf{a}}(\bar{\mathsf{Y}})
	+ \frac{1}{2\beta}\sum_{i,k,j,\ell}\mathcal{H}_{(i,k),(j,\ell)}\Delta y_i^k \Delta y_j^\ell
	+ O(\beta^{-3/2}),
\end{equation*}
uniformly for $\Delta\mathbf{Y}\coloneqq\{\Delta y_i^k\}_{1\le i\le k\le N-1}$
in compact sets (the top row is fixed, so $\Delta y_i^N=0$). Here the linear term
vanishes because $\bar{\mathsf{Y}}$ is a critical point of $H_{\mathsf{a}}$
(\Cref{prop:variational_deformed}),
and the remainder comes from the third-order Taylor terms of the finitely many
logarithmic factors of~\eqref{eq:effective_hamiltonian}, whose arguments stay
bounded away from zero because
$\bar{\mathsf{Y}}$ lies in the open Gelfand--Tsetlin polytope. Hence,
by~\eqref{eq:effective_density}, the density of
$\Delta\mathbf{Y}$, up to a factor independent of $\Delta\mathbf{Y}$,
converges to the deformed dGFF
density~\eqref{eq:deformed_dGFF_density} uniformly on compact sets. Indeed,
the residual factor contributes
$R(\bar{\mathsf{Y}}+\Delta\mathbf{Y}/\sqrt{\beta})/R(\bar{\mathsf{Y}})\to1$,
since $R$ is independent of $\beta$ and is continuous and positive at
$\bar{\mathsf{Y}}$.

The tail bound
(\Cref{lem:gaussian_tail_scaled})
upgrades
this local convergence to convergence in distribution.
Indeed, for a bounded continuous
$\varphi\colon\R^{N(N-1)/2}\to\R$ and $s>0$, split
the expectation
$\E[\varphi(\Delta\mathbf{Y})]$ and the normalization at
the ball $\{\max_{i,k}|\Delta y_i^k|\le s\}$: the
exterior contributes $O(\lVert\varphi\rVert_\infty
e^{-cs^2})$ by~\eqref{eq:gaussian_tail_scaled}, while
inside the ball the unnormalized density converges
uniformly, and the unknown normalizing constant of the
conditional density~\eqref{eq:effective_density} cancels
in the ratio over the common ball.
Letting $\beta\to\infty$ and then $s\to\infty$ identifies
the limit of $\E[\varphi(\Delta\mathbf{Y})]$ as the
expectation of $\varphi$ under the deformed dGFF.
This completes the proof.
\end{proof}

\subsection{The deformed lattice: shifted derivatives for a 
single spike}
\label{sec:shifted_derivative}

In the unperturbed case, the $\infty$-corners lattice is
built level by level starting from the top row
$\mathbf{z}$, and each subsequent level consists of the
roots of the derivative of the polynomial vanishing at
the level above (\Cref{def:deterministic_lattice}).
The \emph{shifted derivative} is defined for $c\in\R$ by
\begin{equation}
	\label{eq:shifted_deriv_operator}
	D_c f(x) \coloneqq f'(x) + c\,f(x) = e^{-cx}\frac{d}{dx}\bigl[e^{cx}f(x)\bigr].
\end{equation}
We prove that for a single spike in the last coordinate, 
that is, for perturbations of the form
$\mathsf{a}_1=\cdots=\mathsf{a}_{N-1}$, with $\mathsf{a}_N$ unrestricted,
the deformed lattice is built exactly as the unperturbed one, except that the
first step, from the top row to level $N-1$, applies $D_{\mathsf{c}}$ with
$\mathsf{c}=\mathsf{a}_{N-1}-\mathsf{a}_N$ in place of the derivative
(\Cref{prop:multilevel_shifted_deriv}).
For all other $\mathsf{a}$, including a single spike at any other coordinate,
the natural analogue of this construction fails
(\Cref{rmk:generic_failure}).

We first record the root structure of $D_c$ on polynomials with distinct real
roots.

\begin{proposition}
	\label{prop:shifted_derivative}
	Let $P(x) = \prod_{j=1}^N(x-\lambda_j)$ with $\lambda_1>\cdots>\lambda_N$,
	and let $c\in\R$; parts (i)--(iv) below assume $c\ne0$, while
	part~(v) treats the case $c=0$.
	\begin{enumerate}[\normalfont(i)]
		\item The polynomial $D_c P(x) = P'(x) + c\,P(x)$ has degree~$N$
		with leading coefficient~$c$.
		Its roots are precisely the critical points of $g(x) = e^{cx}P(x)$.
		\item Exactly $N-1$ of the $N$ roots of $D_c P$ lie in the interlacing
		intervals $(\lambda_{j+1},\lambda_j)$, $j=1,\ldots,N-1$,
		one root per interval.
		\item The remaining root $r_*$ satisfies
		\begin{equation*}
r_* < \lambda_N \;\text{ if } c>0,
			\qquad
			r_* > \lambda_1 \;\text{ if } c<0.
		\end{equation*}
		\item The sum of all $N$ roots equals
		$\sum_{j=1}^N \lambda_j - N/c$.
		\item When $c=0$, the polynomial $D_0P=P'$ has degree $N-1$, and all
		its roots lie in the interlacing intervals, one root per interval;
		there is no root outside $[\lambda_N,\lambda_1]$.
	\end{enumerate}
\end{proposition}

Of the $N$ roots of $D_cP$ for $c\ne0$, we call the $N-1$ roots in the
interlacing intervals the \emph{interlacing roots}, and the remaining root
$r_*$ the \emph{spurious root}. The constructions of this subsection use the
interlacing roots and discard the spurious one; in the 
zero-temperature
up transition of
\Cref{sec:up_transition_scaled}, 
the spurious root plays the role of one of the limiting
positions.

\begin{proof}[Proof of \Cref{prop:shifted_derivative}]
	(i) Since $P$ is monic of degree~$N$, $D_cP = P'+cP$ has degree~$N$
	with leading coefficient~$c$.
	The identity $D_cP = e^{-cx}\frac{d}{dx}[e^{cx}P]$ follows
	from the product rule; since $e^{-cx}\ne 0$, the roots of $D_cP$
	coincide with the zeros of $(e^{cx}P)' = g'(x)$.

	(ii)
	The function $g(x) = e^{cx}P(x)$ vanishes at each $\lambda_j$
	(since $P(\lambda_j)=0$) and $e^{cx}>0$.
	Between consecutive nodes $\lambda_{j+1}$ and $\lambda_j$,
	the polynomial $P$ changes sign at each simple root,
	so $g$ changes sign as well.
	By Rolle's theorem, $g'$ has at least one root in every interval
	$(\lambda_{j+1},\lambda_j)$, giving at least $N-1$ interlacing
	critical points, that is, roots of $D_cP$.
	One further real root lies outside $[\lambda_N,\lambda_1]$.
	Indeed,
	$D_cP(\lambda_j) = P'(\lambda_j)=\prod_{i\ne j}(\lambda_j-\lambda_i)$,
	whose sign is $(-1)^{j-1}$.
	If $c>0$, then $D_cP(x)\sim c\ssp x^N$ has sign $(-1)^N$ as
	$x\to-\infty$, opposite to the sign $(-1)^{N-1}$ of $D_cP(\lambda_N)$,
	so $D_cP$ has a root in $(-\infty,\lambda_N)$;
	if $c<0$, then $D_cP(x)\to-\infty$ as $x\to+\infty$ while
	$D_cP(\lambda_1)>0$, so $D_cP$ has a root in $(\lambda_1,\infty)$.
	This accounts for $N$ real roots; since $D_cP$ has degree~$N$, there
	are no others, and thus exactly one root lies in each interval
	$(\lambda_{j+1},\lambda_j)$ and exactly one outside
	$[\lambda_N,\lambda_1]$.

	(iii) For $c>0$: on $(-\infty,\lambda_N)$,
	$g(x)=e^{cx}P(x)$ has constant sign $(-1)^N$ and satisfies
	$g(x)\to 0$ as $x\to -\infty$ (since $e^{cx}\to 0$ dominates)
	while $g(\lambda_N)=0$.
	Thus $g$ achieves an extremum in $(-\infty,\lambda_N)$,
	placing the spurious critical point~$r_*$ there.
	The case $c<0$ is analogous, with $r_*>\lambda_1$.

	(iv)--(v) follow from the classical Vieta's formulas and the 
	Rolle's theorem argument.
\end{proof}

\begin{proposition}
	\label{prop:multilevel_shifted_deriv}
	Let $N\ge2$, let $\mathsf{a}_1=\cdots=\mathsf{a}_{N-1}$ with
	$\mathsf{a}_N$ arbitrary, and set
	$\mathsf{c} \coloneqq \mathsf{a}_{N-1}-\mathsf{a}_N$.
	Let $P_N(x)=\prod_{j=1}^{N}(x-z_j)$ be the top-row polynomial, and let
	$P_{N-1}$ be the monic polynomial of degree $N-1$ whose roots are the
	$N-1$ interlacing roots of $D_{\mathsf{c}} P_N$.
	Then, for each $k=1,\ldots,N-1$, the level-$k$ positions
	$\bar{\mathsf{y}}_1^k>\cdots>\bar{\mathsf{y}}_k^k$ of the deformed
	lattice are the roots of
	$\left(\frac{d}{dx}\right)^{N-1-k}P_{N-1}$.
	Equivalently, the levels $N-1,\ldots,1$ of the deformed lattice form the
	$\infty$-corners lattice of \Cref{def:deterministic_lattice} whose top
	row consists of the roots of $P_{N-1}$.
	For $\mathsf{c}=0$ we have $P_{N-1}=P_N'/N$, and the deformed lattice is
	the unperturbed one.
\end{proposition}

\begin{proof}
	By \Cref{prop:shifted_derivative}(ii) for
	$\mathsf{c}\ne0$, and (v) for $\mathsf{c}=0$, the roots
	of $P_{N-1}$ are distinct and strictly interlace with
	the top row, one in each gap $(z_{j+1},z_j)$. The roots
	of the successive derivatives interlace strictly by
	Rolle's theorem. Hence the array lies in the open
	Gelfand--Tsetlin polytope. By the uniqueness in
	\Cref{prop:variational_deformed}, it suffices to check
	that the array solves the coupled
	equations~\eqref{eq:deformed_optimality}.

	The levels $N-1,\ldots,1$ of the array form, by construction, the
	unperturbed $\infty$-corners lattice of \Cref{def:deterministic_lattice}
	with top row the roots of $P_{N-1}$, so \Cref{lem:polynomial_identities}
	applies to them. To identify the array with the deformed lattice, we
	compare~\eqref{eq:polynomial_identity3}
	with~\eqref{eq:deformed_optimality}: the same-level sum
	of~\eqref{eq:deformed_optimality} cancels against the lower-level one at
	every level (at $k=1$ both are empty). The equation at level $k$ thus
	reduces to
	\begin{equation*}
		-\frac12\sum_{q=1}^{k+1}
		\frac{1}{\bar{\mathsf{y}}_i^k-\bar{\mathsf{y}}_q^{k+1}}
		=\frac{\mathsf{a}_k-\mathsf{a}_{k+1}}{2},
		\qquad i=1,\ldots,k.
	\end{equation*}
	For $k\le N-2$ both sides vanish: the right-hand side because
	$\mathsf{a}_k=\mathsf{a}_{k+1}$, and the left-hand side because the
	level-$k$ points are the critical points of the monic polynomial
	$P_{k+1}$ with roots $\bar{\mathsf{y}}^{k+1}$, so that
	$\sum_{q}(\bar{\mathsf{y}}_i^k-\bar{\mathsf{y}}_q^{k+1})^{-1}
	=(P_{k+1}'/P_{k+1})(\bar{\mathsf{y}}_i^k)=0$.
	For $k=N-1$ the reduced equation reads
	$\sum_{j=1}^{N}(\bar{\mathsf{y}}_i^{N-1}-z_j)^{-1}=-\mathsf{c}$, that is,
	$P_N'(\bar{\mathsf{y}}_i^{N-1})
	+\mathsf{c}\ssp P_N(\bar{\mathsf{y}}_i^{N-1})=0$,
	which holds because $\bar{\mathsf{y}}_i^{N-1}$ is a root of
	$D_{\mathsf{c}}P_N$.
\end{proof}

\begin{remark}
	\label{rmk:generic_failure}
	For all other $\mathsf{a}$ --- those with
	$\mathsf{a}_k\ne\mathsf{a}_{k+1}$ for some $k\le N-2$,
	for instance a single spike at a coordinate other than
	the last (for $N\ge3$) --- the natural analogue of this
	construction fails.  The
	candidate is the top-down recursion that takes, at each
	level, the interlacing roots of $D_{\mathsf{a}_k-\mathsf{a}_{k+1}}$
	applied to the polynomial vanishing at the level above.
	It produces an interlacing array
	$\hat{\mathsf{Y}}=\{\hat{\mathsf{y}}_i^k\}$, but
	$\hat{\mathsf{Y}}\ne\bar{\mathsf{Y}}$.
	The latter may be verified by a direct
	computation, which we omit.
\end{remark}

% The omitted computation, kept here for posterity (not typeset):
%
% Write d_k = a_k - a_{k+1} and let P^_k be the monic polynomial vanishing
% at level k of Y^.  Suppose d_m \ne 0 for some m \le N-2, and put k=m+1,
% so that 2 \le k \le N-1.  Let x be one of the k roots of P^_k.  The
% recursion on the bond above gives P^_{k+1}'(x) + d_k P^_{k+1}(x) = 0,
% for d_k = 0 as well as for d_k \ne 0.  The bond below reads
%     P^_k'(u) + d_m P^_k(u) = d_m (u - r_*) P^_{k-1}(u),
% where r_* is the spurious root of D_{d_m} P^_k; differentiating in u and
% evaluating at u = x gives
%     P^_{k-1}'/P^_{k-1} = P^_k''/P^_k' + d_m - (x - r_*)^{-1}.
% Substituting both into \eqref{eq:deformed_optimality}, whose three sums at
% level k are, in terms of P^ evaluated at x,
%     sum_{j\ne i} 1/(x - y^_j^k)     = (1/2) P^_k''/P^_k',
%     sum_{q=1}^{k+1} 1/(x - y^_q^{k+1}) = P^_{k+1}'/P^_{k+1},
%     sum_{p=1}^{k-1} 1/(x - y^_p^{k-1}) = P^_{k-1}'/P^_{k-1},
% the left-hand side of \eqref{eq:deformed_optimality} minus its right-hand
% side collapses to
%     (1/2) ( (x - r_*)^{-1} - d_m ).
% Since x -> (x - r_*)^{-1} is injective, this vanishes for at most one of
% the k \ge 2 points of level k.  Hence Y^ does not solve
% \eqref{eq:deformed_optimality}, and by the uniqueness in
% \Cref{prop:variational_deformed} it differs from Ybar.
%
% Checked numerically: z=(1,0,-1), a=(3,0,0) gives level 2 at +-1/sqrt(3),
% level 1 at 1/3, spurious root r_*=-1, and residual (-3-sqrt3)/4 \ne 0 at
% x=1/sqrt(3).  Verified again independently on three N=4 configurations.

\section{Asymptotics of up transitions}
\label{sec:up_transitions}

The random polynomial equation
\eqref{eq:random_poly_general}
determining the conditional distribution
of the upper row $\boldsymbol{\lambda}$ given the lower row
$\boldsymbol{\mu}$, $\boldsymbol{\lambda}\succ\boldsymbol{\mu}$,
runs in the
direction opposite to \Cref{thm:crystallization_fixed} and
\Cref{thm:crystallization_scaled}.
In this section, we discuss the 
zero-temperature limit $\beta\to\infty$ of this conditional
density in various regimes, in the fixed and 
scaled perturbation settings.
Throughout the section $P_{\boldsymbol{\mu}}(x)=\prod_{i=1}^{N-1}(x-\mu_i)$,
and we assume $N\ge2$.

\subsection{Fixed perturbation}
\label{sec:up_transition_fixed}

\begin{proposition}
\label{prop:poly_limit_fixed}
Fix $a_N\in\R$, and let random $\boldsymbol{\lambda}$ be
obtained from a fixed $\boldsymbol{\mu}$ through~\eqref{eq:random_poly_general}.
Then, as $\beta\to\infty$:
\begin{enumerate}[\normalfont(i)]
	\item 
	If $\boldsymbol{\mu}$ does not depend on $\beta$, then
	the $N-2$ inner roots $\lambda_2,\ldots,\lambda_{N-1} $
	converge in probability to the roots of
	$P_{\boldsymbol{\mu}}'$,
	\begin{equation}
		\label{eq:limit_poly_eq_fixed}
		\sum_{i=1}^{N-1}\frac{1}{x-\mu_i}=0,
	\end{equation}
	one per gap $(\mu_{i+1},\mu_i)$, while the two outer roots escape
	(both limits in probability):
	\begin{equation*}
		\frac{\lambda_1}{\sqrt{\beta}}\to\sqrt{\frac{N-1}{2}},
		\qquad
		\frac{\lambda_N}{\sqrt{\beta}}\to-\sqrt{\frac{N-1}{2}}.
	\end{equation*}
	\item
	If instead $\mu_i=\sqrt{\beta}\ssp m_i$ with
	$\mathbf{m}=(m_1>\cdots>m_{N-1})$ fixed, then no root escapes: the whole
	rescaled row $\boldsymbol{\lambda}/\sqrt{\beta}$ converges in probability to
	the $N$ roots of
	\begin{equation}
		\label{eq:gaussian_shifted_derivative}
		P_{\mathbf{m}}'(x)-2x\ssp P_{\mathbf{m}}(x)
		= e^{x^{2}}\frac{d}{dx}\Bigl[e^{-x^{2}}P_{\mathbf{m}}(x)\Bigr]=0,
	\end{equation}
	and these $N$ roots interlace with $\mathbf{m}$.
\end{enumerate}
\end{proposition}

On Hermite polynomials $H_n$, the up transition of~(ii) is the Rodrigues
(raising) step
$$e^{x^{2}}\frac{d}{dx}\bigl[e^{-x^{2}}H_{n-1}\bigr]=-H_{n},$$ taking the roots
of $H_{n-1}$ to the roots of $H_n$. This agrees with the $\infty$-corners
lattice of~\cite{GorinMarcus2020}: for the top row at the roots of $H_N$, the
lowering relation $H_n'=2nH_{n-1}$ places level $k$ at the roots of $H_k$ for
every $k$, and the up transition climbs the same tower.

\begin{proof}[Proof of \Cref{prop:poly_limit_fixed}]
Since $\xi_i\sim\frac12\chi^2_\beta$ has mean $\beta/2$ and variance $\beta/2$,
Chebyshev's inequality gives $\xi_i/(\beta/2)\to1$ in probability for
$i=1,\ldots,N-1$; and $\xi_N\sim\mathcal{N}(a_N,1)$ is $O_P(1)$. Both statements
are unconditional, the law of $\boldsymbol{\xi}$ in~\eqref{eq:xi_distributions}
being independent of $\boldsymbol{\mu}$.

(i) Dividing~\eqref{eq:random_poly_general} by $\beta/2$ gives
\begin{equation*}
	\sum_{i=1}^{N-1}\frac{2\xi_i/\beta}{x-\mu_i}=\frac{2}{\beta}\bigl(x-\xi_N\bigr).
\end{equation*}
On compact sets at positive distance from $\mu_1,\ldots,\mu_{N-1}$ the
left-hand side becomes $\sum_i(x-\mu_i)^{-1}$, while the right-hand side goes to
$0$, both uniformly and in probability. This
gives~\eqref{eq:limit_poly_eq_fixed}. Its solutions are the $N-2$ roots of
$P_{\boldsymbol{\mu}}'$, one in each gap $(\mu_{i+1},\mu_i)$ by Rolle's theorem.
Multiplying the last display by $P_{\boldsymbol{\mu}}(x)$ turns it into a
polynomial equation of degree~$N$ whose roots are $\lambda_1,\ldots,\lambda_N$.
The coefficients of this polynomial converge in probability to those of
$P_{\boldsymbol{\mu}}'$, in particular the two leading ones to~$0$.
So $N-2$ roots converge to the roots of $P_{\boldsymbol{\mu}}'$, and the other
two leave every compact set.
By the interlacing $\boldsymbol{\mu}\prec\boldsymbol{\lambda}$, the former are
$\lambda_2,\ldots,\lambda_{N-1}$.
For the remaining two, solve~\eqref{eq:poly_identity_general} for
$\prod_{j=1}^{N}(x-\lambda_j)$ and substitute $x=\sqrt{\beta}\ssp t$: the monic
polynomial $\prod_{j=1}^{N}\bigl(t-\lambda_j/\sqrt{\beta}\bigr)$ converges in
probability, coefficientwise, to $t^{N-2}\bigl(t^{2}-\tfrac{N-1}{2}\bigr)$, so
exactly two roots reach scale~$\sqrt{\beta}$, with limits
$\pm\sqrt{(N-1)/2}$.

(ii) With $\mu_i=\sqrt{\beta}\ssp m_i$ and $x=\sqrt{\beta}\ssp\ell$, equation
\eqref{eq:random_poly_general} reads
\begin{equation*}
	\frac{1}{\sqrt{\beta}}\sum_{i=1}^{N-1}\frac{\xi_i}{\ell-m_i}
	=\sqrt{\beta}\ssp\ell-\xi_N .
\end{equation*}
Dividing by $\sqrt{\beta}$ and using $\xi_i/(\beta/2)\to1$ together with
$\xi_N/\sqrt{\beta}\to0$ turns the last display into
$\frac12\sum_i(\ell-m_i)^{-1}=\ell$, which is~\eqref{eq:gaussian_shifted_derivative}
after multiplying by $2P_{\mathbf{m}}(\ell)$.
The limiting polynomial equation \eqref{eq:gaussian_shifted_derivative}
has degree $N$ with leading coefficient $-2$, so it has
$N$ roots. They interlace with $\mathbf{m}$: the function $e^{-x^{2}}P_{\mathbf{m}}(x)$
vanishes at $m_1,\ldots,m_{N-1}$ and tends to $0$ at $\pm\infty$, so by Rolle's
theorem its derivative vanishes once in each of the $N-2$ gaps, once above $m_1$
and once below $m_{N-1}$.
Convergence of the roots follows from that of the coefficients, since the leading
coefficient is bounded away from $0$.
This completes the proof.
\end{proof}

The fixed perturbation is invisible in both regimes of \Cref{prop:poly_limit_fixed},
for the same reason as in
\Cref{thm:crystallization_fixed}: $a_N$ enters only through
$\xi_N\sim\mathcal{N}(a_N,1)$, which is small compared to other coefficients.

\subsection{Linearly growing perturbation}
\label{sec:up_transition_scaled}

The difference from \Cref{sec:up_transition_fixed} is
that $\xi_N$ now contributes at leading order, and the ordinary derivative
of~\eqref{eq:limit_poly_eq_fixed} is replaced by the shifted derivative
$D_{\mathsf{a}_N}$~\eqref{eq:shifted_deriv_operator}, whose root structure is
recorded in \Cref{prop:shifted_derivative}.

\begin{proposition}
\label{prop:poly_limit_scaled}
Let $a_N=\frac{\beta}{2}\mathsf{a}_N$ with $\mathsf{a}_N\ne0$ fixed, let
$\boldsymbol{\mu}=(\mu_1>\cdots>\mu_{N-1})$ be fixed and not depend on $\beta$, and let random
$\boldsymbol{\lambda}$ be obtained from $\boldsymbol{\mu}$
through~\eqref{eq:random_poly_general}. Then, as $\beta\to\infty$:
\begin{enumerate}[\normalfont(i)]
	\item $N-1$ of the roots stay bounded and converge in probability to the roots of
	\begin{equation}
		\label{eq:limit_poly_eq_scaled}
		\sum_{i=1}^{N-1}\frac{1}{x-\mu_i}=-\mathsf{a}_N,
		\qquad\text{that is,}\qquad
		D_{\mathsf{a}_N}P_{\boldsymbol{\mu}}(x)
		= P_{\boldsymbol{\mu}}'(x)+\mathsf{a}_N P_{\boldsymbol{\mu}}(x)=0 .
	\end{equation}
	\Cref{prop:shifted_derivative}, applied to $P_{\boldsymbol{\mu}}$, locates
	these $N-1$ roots: exactly $N-2$ of them interlace with $\boldsymbol{\mu}$,
	one in each gap $(\mu_{i+1},\mu_i)$, and the last lies outside the range of
	$\boldsymbol{\mu}$ --- below $\mu_{N-1}$ if $\mathsf{a}_N>0$, above $\mu_1$
	if $\mathsf{a}_N<0$.
	\item The remaining root escapes on the scale of the perturbation:
	$\dfrac{2}{\beta}\max_j|\lambda_j|\to|\mathsf{a}_N|$,
	the unbounded root having the sign of $\mathsf{a}_N$.
\end{enumerate}
\end{proposition}

We call the root in~(ii) the \emph{escaping root}
and denote it by $\lambda_{\mathrm{esc}}$.
The root
of~\eqref{eq:limit_poly_eq_scaled} outside the range of $\boldsymbol{\mu}$
in~(i) is referred to as
the \emph{spurious root}, as in \Cref{prop:shifted_derivative}.

\begin{proof}[Proof of \Cref{prop:poly_limit_scaled}]
As in the proof of \Cref{prop:poly_limit_fixed},
$\xi_i/(\beta/2)\to1$ in probability for $i=1,\ldots,N-1$; and
$\xi_N\sim\mathcal{N}(\frac{\beta}{2}\mathsf{a}_N,1)$, so
$\frac{2}{\beta}\xi_N=\mathsf{a}_N+O_P(\beta^{-1})$.

(i) Dividing~\eqref{eq:random_poly_general} by $\beta/2$ gives
\begin{equation*}
	\sum_{i=1}^{N-1}\frac{2\xi_i/\beta}{x-\mu_i}
	=\frac{2x}{\beta}-\frac{2\xi_N}{\beta}.
\end{equation*}
On any compact set at positive distance from $\mu_1,\ldots,\mu_{N-1}$, the
left-hand side converges to $\sum_i\frac{1}{x-\mu_i}$ in probability, while
the right-hand side converges to $-\mathsf{a}_N$; multiplying
by $P_{\boldsymbol{\mu}}(x)$ turns~\eqref{eq:limit_poly_eq_scaled} into
$D_{\mathsf{a}_N}P_{\boldsymbol{\mu}}=0$.
By \Cref{prop:shifted_derivative}, exactly
$N-1$ of the roots 
of this equation
stay bounded. That they converge
in probability to the deterministic roots
of \eqref{eq:limit_poly_eq_scaled}
follows
similarly to the proof of \Cref{prop:poly_limit_fixed}\,(i).

(ii) The row $\boldsymbol{\lambda}$ has $N$ points and~(i) accounts for $N-1$
of them, which stay bounded in probability. 
We denoted the remaining root by
$\lambda_{\mathrm{esc}}$.
From the relation
\eqref{eq_xi2}
between
$|\boldsymbol{\lambda}|$, 
$|\boldsymbol{\mu}|$, and $\xi_N$,
we see that
the $N-1$ bounded roots and $|\boldsymbol{\mu}|$ contribute $O_P(1)$, while
$\frac{2}{\beta}\ssp\xi_N\to\mathsf{a}_N$; hence
$\frac{2}{\beta}\ssp\lambda_{\mathrm{esc}}\to\mathsf{a}_N$ in probability,
which gives the claimed scale together with the sign of the escaping root.
\end{proof}

\begin{remark}
	The two extreme roots $\lambda_1$ and $\lambda_N$ --- the two lying
	outside the range of $\boldsymbol{\mu}$ --- are governed in
	\Cref{prop:poly_limit_fixed}\,(i) and \Cref{prop:poly_limit_scaled} by a
	common approximate
	equation. Indeed, replacing $\sum_i\xi_i/(x-\mu_i)$ by $\frac{\beta(N-1)}{2x}$
	and $\xi_N$ by $\frac{\beta}{2}\mathsf{a}_N$
	in~\eqref{eq:random_poly_general} yields
	\begin{equation*}
		x^{2}-\frac{\beta\mathsf{a}_N}{2}\ssp x-\frac{\beta(N-1)}{2}=0 .
	\end{equation*}
	At $\mathsf{a}_N=0$ its roots form the symmetric pair
	$\pm\sqrt{\beta(N-1)/2}$, matching the limits of $\lambda_1$ and
	$\lambda_N$ in \Cref{prop:poly_limit_fixed}\,(i); for fixed
	$\mathsf{a}_N\ne0$ they are approximately $\frac{\beta\mathsf{a}_N}{2}$
	and $-\frac{N-1}{\mathsf{a}_N}$: the former matches the escaping root of
	\Cref{prop:poly_limit_scaled}\,(ii), and the latter approaches the
	spurious root of \Cref{prop:poly_limit_scaled}\,(i) as
	$\mathsf{a}_N\to0$.
	The two regimes cross over at $\mathsf{a}_N\asymp\beta^{-1/2}$, where
	both extreme roots are at scale~$\sqrt{\beta}$.
\end{remark}


\medskip

\textsc{l. petrov, university of virginia, charlottesville, va, usa}

e-mail: \texttt{lenia.petrov@gmail.com}

\medskip

\textsc{j. xu, the ohio state university, columbus, oh, usa}

e-mail: \texttt{jxu0800@gmail.com}


\begin{bibdiv}
\begin{biblist}
\bib{adler2013random}{article}{
      author={Adler, M.},
      author={van Moerbeke, P.},
      author={Wang, D.},
       title={Random matrix minor processes related to percolation theory},
        date={2013},
     journal={Random Matrices Theory Appl.},
      volume={2},
      number={04},
       pages={1350008},
        note={arXiv:1301.7017 [math.PR]},
}

\bib{Anderson1991Selberg}{article}{
      author={Anderson, G.~W.},
       title={{A short proof of {Selberg}'s generalized beta formula}},
        date={1991},
     journal={Forum Math.},
      volume={3},
       pages={415\ndash 417},
}

\bib{BBP2005phase}{article}{
      author={Baik, J.},
      author={{Ben Arous}, G.},
      author={P\'ech\'e, S.},
       title={{Phase transition of the largest eigenvalue for nonnull complex
  sample covariance matrices}},
        date={2005},
     journal={Ann. Probab.},
      volume={33},
      number={5},
       pages={1643\ndash 1697},
        note={arXiv:math/0403022 [math.PR]},
}

\bib{Baryshnikov_GUE2001}{article}{
      author={Baryshnikov, Yu.},
       title={{GUEs and queues}},
        date={2001},
     journal={Probab. Theory Relat. Fields},
      volume={119},
       pages={256\ndash 274},
}

\bib{benaych2022matrix}{article}{
      author={Benaych-Georges, F.},
      author={Cuenca, C.},
      author={Gorin, V.},
       title={{Matrix addition and the Dunkl transform at high temperature}},
        date={2022},
     journal={Commun. Math. Phys.},
      volume={394},
      number={2},
       pages={735\ndash 795},
        note={arXiv:2105.03795 [math-ph]},
}

\bib{bloemendal2013limits}{article}{
      author={Bloemendal, A.},
      author={Vir{\'a}g, B.},
       title={{Limits of spiked random matrices I}},
        date={2013},
     journal={Probab. Theory Relat. Fields},
      volume={156},
       pages={795\ndash 825},
        note={arXiv:1011.1877 [math.PR]},
}

\bib{bloemendal2016limits}{article}{
      author={Bloemendal, A.},
      author={Vir{\'a}g, B.},
       title={{Limits of spiked random matrices II}},
        date={2016},
     journal={Ann. Probab.},
      volume={44},
       pages={2726\ndash 2769},
        note={arXiv:1109.3704 [math.PR]},
}

\bib{BorodinGorin2015}{article}{
      author={Borodin, A.},
      author={Gorin, V.},
       title={{General beta-{Jacobi} corners process and the {Gaussian} free
  field}},
        date={2015},
     journal={Comm. Pure Appl. Math.},
      volume={68},
      number={10},
       pages={1774\ndash 1844},
        note={arXiv:1305.3627 [math.PR]},
}

\bib{cuenca2025discrete}{article}{
      author={Cuenca, C.},
      author={Do{\l}{\k e}ga, M.},
       title={{Discrete $N$-particle systems at high temperature through Jack
  generating functions}},
        date={2025},
     journal={arXiv preprint},
        note={arXiv:2502.13098 [math.PR]},
}

\bib{DesrosiersLiu2015}{article}{
      author={Desrosiers, P.},
      author={Liu, D.-Z.},
       title={{Scaling limits of correlations of characteristic polynomials for
  the {Gaussian} $\beta$-ensemble with external source}},
        date={2015},
     journal={Int. Math. Res. Not. IMRN},
      volume={2015},
      number={12},
       pages={3751\ndash 3781},
        note={arXiv:1306.4058 [math-ph]},
}

\bib{dixon1905generalization}{article}{
      author={Dixon, A.~L.},
       title={{Generalization of Legendre's Formula $KE' - (K - E)K' =
  \frac{1}{2}\pi$}},
        date={1905},
     journal={Proc. Lond. Math. Soc.},
      volume={3},
      number={1},
       pages={206\ndash 224},
}

\bib{dumitriu2002matrix}{article}{
      author={Dumitriu, I.},
      author={Edelman, A.},
       title={{Matrix models for beta ensembles}},
        date={2002},
     journal={J. Math. Phys.},
      volume={43},
      number={11},
       pages={5830\ndash 5847},
        note={arXiv:math-ph/0206043},
}

\bib{dyson1962brownian}{article}{
      author={Dyson, F.~J.},
       title={{A Brownian-motion model for the eigenvalues of a random
  matrix}},
        date={1962},
     journal={J. Math. Phys.},
      volume={3},
      number={6},
       pages={1191\ndash 1198},
}

\bib{Ferrari2014PerturbedGUE}{article}{
      author={Ferrari, P.~L.},
      author={Frings, R.},
       title={{Perturbed GUE minor process and Warren's process with drifts}},
        date={2014},
     journal={J. Stat. Phys.},
      volume={154},
      number={1-2},
       pages={356\ndash 377},
        note={arXiv:1212.5534 [math-ph]},
}

\bib{Forrester-LogGas}{book}{
      author={Forrester, P.~J.},
       title={{Log-gases and random matrices}},
   publisher={Princeton University Press},
        date={2010},
}

\bib{Forrester2013source}{article}{
      author={Forrester, P.~J.},
       title={{The averaged characteristic polynomial for the {Gaussian} and
  chiral {Gaussian} ensembles with a source}},
        date={2013},
     journal={J. Phys. A},
      volume={46},
      number={34},
       pages={345204},
        note={arXiv:1203.5838 [math-ph]},
}

\bib{LamarreShkolnikov2019}{article}{
      author={Gaudreau~Lamarre, P.~Y.},
      author={Shkolnikov, M.},
       title={{Edge of spiked beta ensembles, stochastic {Airy} semigroups and
  reflected {Brownian} motions}},
        date={2019},
     journal={Ann. Inst. Henri Poincar\'e Probab. Stat.},
      volume={55},
      number={3},
       pages={1402\ndash 1438},
        note={arXiv:1706.08451 [math.PR]},
}

\bib{GorinKleptsyn2024}{article}{
      author={Gorin, V.},
      author={Kleptsyn, V.},
       title={{Universal objects of the infinite beta random matrix theory}},
        date={2024},
     journal={J. Eur. Math. Soc.},
      volume={26},
      number={9},
       pages={3429\ndash 3496},
        note={arXiv:2009.02006 [math.PR]},
}

\bib{GorinMarcus2020}{article}{
      author={Gorin, V.},
      author={Marcus, A.~W.},
       title={{Crystallization of random matrix orbits}},
        date={2020},
     journal={IMRN},
      volume={2020},
      number={3},
       pages={883\ndash 913},
        note={arXiv:1706.07393 [math.PR]},
}

\bib{GorinShkolnikov2014}{article}{
      author={Gorin, V.},
      author={Shkolnikov, M.},
       title={{Multilevel Dyson Brownian motions via Jack polynomials}},
        date={2015},
     journal={Probab. Theory Relat. Fields},
      volume={163},
      number={3-4},
       pages={413\ndash 463},
        note={arXiv:1401.5595 [math.PR]},
}

\bib{gorin2024airy}{article}{
      author={Gorin, V.},
      author={Xu, J.},
      author={Zhang, L.},
       title={{Airy$_\beta$ line ensemble and its Laplace transform}},
        date={2024},
     journal={arXiv preprint},
        note={arXiv:2411.10829 [math.PR]},
}

\bib{GuhrKohler2002}{article}{
      author={Guhr, T.},
      author={Kohler, H.},
       title={{Recursive construction for a class of radial functions. {I}.
  {O}rdinary space}},
        date={2002},
     journal={J. Math. Phys.},
      volume={43},
      number={5},
       pages={2707\ndash 2740},
        note={arXiv:math-ph/0011007 [math-ph]},
}

\bib{HarishChandra1957}{article}{
      author={Harish-Chandra},
       title={{Differential operators on a semisimple {Lie} algebra}},
        date={1957},
     journal={Amer. J. Math.},
      volume={79},
      number={1},
       pages={87\ndash 120},
}

\bib{HeckmannOpdam1997}{article}{
      author={Heckman, G.J.},
      author={Opdam, E.M.},
       title={{Yang's system of particles and Hecke algebras}},
        date={1997},
     journal={Ann. of Math.},
      volume={145},
      number={1},
       pages={139\ndash 173},
}

\bib{itzykson1980planar}{article}{
      author={Itzykson, C.},
      author={Zuber, J.-B.},
       title={{The planar approximation. II}},
        date={1980},
     journal={J. Math. Phys.},
      volume={21},
      number={3},
       pages={411\ndash 421},
}

\bib{johansson2006eigenvalues}{article}{
      author={Johansson, K.},
      author={Nordenstam, E.},
       title={{Eigenvalues of GUE minors}},
        date={2006},
     journal={Electron. J. Probab.},
      volume={11},
      number={50},
       pages={1342\ndash 1371},
        note={arXiv:math/0606760 [math.PR]; Erratum: Electron. J. Probab. 12
  (2007), no.~37, 1048--1051},
}

\bib{KeatingXu2024}{article}{
      author={Keating, D.},
      author={Xu, J.},
       title={{Airy limit for $\beta$-additions through {Dunkl} operators}},
        date={2024},
     journal={arXiv preprint},
        note={arXiv:2411.12149 [math.PR]},
}

\bib{LernerBrecher2023}{article}{
      author={Lerner-Brecher, M.},
       title={{On the hard edge limit of the zero temperature {Laguerre}
  $\beta$-corners process}},
        date={2025},
     journal={Ann. Inst. Henri Poincar\'e Probab. Stat.},
      volume={61},
      number={4},
       pages={2721\ndash 2747},
        note={arXiv:2308.14707 [math.PR]},
}

\bib{Marcus2021}{article}{
      author={Marcus, A.~W.},
       title={{Polynomial convolutions and (finite) free probability}},
        date={2021},
     journal={arXiv preprint},
        note={arXiv:2108.07054 [math.CO]},
}

\bib{MarcusSpielmanSrivastava2022}{article}{
      author={Marcus, A.~W.},
      author={Spielman, D.~A.},
      author={Srivastava, N.},
       title={{Finite free convolutions of polynomials}},
        date={2022},
     journal={Probab. Theory Relat. Fields},
      volume={182},
       pages={807\ndash 848},
        note={arXiv:1504.00350 [math.CO]},
}

\bib{Neretin2003Rayleigh}{article}{
      author={Neretin, Yu.A.},
       title={{Rayleigh triangles and non-matrix interpolation of matrix beta
  integrals}},
        date={2003},
     journal={Sb. Math.},
      volume={194},
      number={4},
       pages={515\ndash 540},
        note={arXiv:math/0301070},
}

\bib{OkounkovOlshanskiJack1996}{article}{
      author={Okounkov, A.},
      author={Olshanski, G.},
       title={{Shifted {Jack} polynomials, binomial formula, and
  applications}},
        date={1997},
        ISSN={1073-2780},
     journal={Math. Res. Lett.},
      volume={4},
      number={1},
       pages={69\ndash 78},
        note={arXiv:q-alg/9608020},
}

\bib{Opdam1993}{article}{
      author={Opdam, E.~M.},
       title={{Dunkl operators, {B}essel functions and the discriminant of a
  finite {C}oxeter group}},
        date={1993},
     journal={Compos. Math.},
      volume={85},
      number={3},
       pages={333\ndash 373},
         url={http://www.numdam.org/item/CM_1993__85_3_333_0/},
}

\bib{Peche2006}{article}{
      author={P\'ech\'e, S.},
       title={{The largest eigenvalue of small rank perturbations of
  {Hermitian} random matrices}},
        date={2006},
     journal={Probab. Theory Relat. Fields},
      volume={134},
      number={1},
       pages={127\ndash 173},
        note={arXiv:math/0411487 [math.PR]},
}

\bib{PetrovTikhonov2019}{article}{
      author={Petrov, L.},
      author={Tikhonov, M.},
       title={{Parameter symmetry in perturbed GUE corners process and
  reflected drifted Brownian motions}},
        date={2020},
     journal={J. Stat. Phys.},
      volume={181},
      number={5},
       pages={1996\ndash 2010},
        note={arXiv:1912.08671 [math.PR]},
}

\bib{warren2005dyson}{article}{
      author={Warren, J.},
       title={{Dyson's Brownian motions, intertwining and interlacing}},
        date={2007},
     journal={Electron. J. Probab.},
      volume={12},
      number={19},
       pages={573\ndash 590},
        note={arXiv:math/0509720 [math.PR]},
}

\bib{Wong2001}{book}{
      author={Wong, R.},
       title={{Asymptotic Approximations of Integrals}},
      series={Classics in Applied Mathematics},
   publisher={SIAM},
        date={2001},
      volume={34},
}
\end{biblist}
\end{bibdiv}
\end{document}